\documentclass{amsart}
\usepackage{amssymb}

\usepackage{stmaryrd}

\usepackage{graphicx}
\usepackage{subfig}
\usepackage[nobysame]{amsrefs}
\usepackage{xcolor}
\usepackage{theoremref}
\usepackage{hyperref}
\newtheorem{theorem}{Theorem}[section]
\newtheorem{lemma}[theorem]{Lemma}
\newtheorem{co}[theorem]{Corollary}
\newtheorem{prop}[theorem]{Proposition}

\theoremstyle{definition}

\newtheorem{ex}[theorem]{Example}

\theoremstyle{remark}
\newtheorem{remark}[theorem]{Remark}

\numberwithin{equation}{section}

\begin{document}
\title[Unique continuation for Schrödinger equations]{Solving the unique continuation problem for Schrödinger equations with low regularity solutions using a stabilized finite element method}

\author{Erik Burman}
\address{Department of Mathematics, University College London, UK}

\author{Mingfei Lu}
\address{Department of Mathematics, University College London, UK}

\author{Lauri Oksanen}
\address{Department of Mathematics and Statistics, University of Helsinki, Finland}

\begin{abstract} 
In this paper, we consider the unique continuation problem for the Schrödinger equations. We prove a Hölder type conditional stability estimate and build up a parameterized stabilized finite element scheme adaptive to the \textit{a priori} knowledge of the solution, achieving error estimates in interior domains with convergence up to continuous stability. The approximability of the scheme to solutions with regularity of as low as $H^1$ is studied and the convergence rates for discrete solutions under $L^2$ and $H^1$ norms are shown. Comparisons in terms of different parameterization for different regularities will be illustrated with respect to the convergence and condition numbers of the linear systems. Finally, numerical experiments will be given to illustrate the theory.

\end{abstract}
\maketitle
\section{Introduction}
We will consider the \textbf{Unique Continuation Problem} for Schrödinger equations
\begin{equation*}
    -\Delta u + Pu = f,
\end{equation*}
and develop a stabilized finite element method to solve it numerically. Similar work for Poisson's equation can be found in \cites{MR3134434,burman2018solving}, and for Helmholtz equation in \cite{burman2019unique}. In all these works no regularization of the continuous problem is introduced, instead numerical stability and the stability of the continuous problem are handled separately. For numerical stability, a primal-dual framework is considered, with stabilization of certain residuals. To quantify the stability of the unique continuation problems on the other hand conditional stability estimates are applied. The design criteria of the stabilization terms is then to obtain the best convergence order possible, while minimizing the effect of perturbations in data.  These idea were first introduced in \cite{Bur14CR}. A related earlier approach was the quasi-reversibility method on mixed form introduced in \cites{Bour05} where Tikhonov regularization on the continuous level was introduced to ensure a well-posed problem. In particular in \cites{Bour05} the convergence for vanishing regularization was proven assuming only that the solution is in $H^1(\Omega)$ for unperturbed data. The difficulty with this approach under discretization is that it appears to be impossible to match the rate of the vanishing regularization keeping optimal convergence of the approximation method. The analysis of \cites{MR3134434,Bur14CR, burman2018solving, burman2019unique} typically requires stronger regularity assumptions in order to show that the numerical method has a convergence order. The objective of the present work is to show how to design the stabilized methods so that low regularity can be handled in a similar fashion, reflecting the conditional stability and the approximation of the finite element space in the regularity class of the exact solution. Observe that the design of the method requires a priori knowledge of this regularity to obtain the best convergence rate. Similar estimates combining conditional stability and approximation have recently been obtained in \cites{DMS23} in the framework of dual norm least squares methods.

The mathematical formulation of the unique continuation problem for Schrödinger equations is as follows: 
Let $\omega\subset \Omega\subset \mathbb{R}^n$ be domains. Let $P\in L^\infty(\Omega)$ be the potential. Find $u\in H^1(\Omega)$, such that 
\begin{equation}\label{pbl}
    \left\{\begin{array}{rcc}
         -\Delta u+Pu& =f & \texttt{in $\Omega$} \\
         u&  =q & \texttt{in $\omega$}
    \end{array}.
    \right.
\end{equation}
It is shown in Chapter \ref{sec:continuous} that if a solution exists, then it is unique. However, since such unique continuation problem are known to be ill-posed, in terms of continuity, we can at most expect a conditional stability estimate that reads:
\begin{equation*}
    ||u||_{H^1(B)}\le C(P)(||-\Delta u+Pu||_{H^{-1}(\Omega)}+||u||_{L^2(\omega)})^\kappa (||-\Delta u+Pu||_{H^{-1}(\Omega)}+||u||_{H^1(\Omega)})^{1-\kappa},
\end{equation*}
and 
\begin{equation*}
    ||u||_{L^2(B)}\le C(P)(||-\Delta u+Pu||_{H^{-2}(\Omega)}+||u||_{L^2(\omega)})^\kappa (||-\Delta u+Pu||_{H^{-2}(\Omega)}+||u||_{L^2(\Omega)})^{1-\kappa},
\end{equation*}
for $B \setminus \omega \subset\subset\Omega$ and $\kappa\in(0,1)$. When $B\setminus \omega$ touches the boundary, the best estimate achieved is logarithmic, see \cite{john1960continuous}. In this paper, we will focus on the former case, i.e. the area of interest $\overline{B\setminus \omega}\subset \Omega$.\\

Our aim is to build a stabilized finite element method to solve (\ref{pbl}) computationally. With the minimal a priori knowledge $u\in H^1(\Omega)$ we show that the scheme converges with a rate matching the conditional stability in $L^2$ norm, but only weakly in $H^1$ norm without a convergence rate. Such low regularity can result from three different reasons:\\

1. low regularity of the right-hand side $f$;\\

2. low regularity of the geometry $\Omega$ on segments of the boundary that are actual limits of the computational domain in physical space, but with unspecified boundary data;\\

3. low regularity of such unknown boundary data.\\

Note that geometric singularities appear only if $\partial \Omega$ coincides with a physical boundary of the domain with unknown boundary conditions and beyond which the solution can not be smoothly extended. A corner on a part of the domain that is an artificial truncation of the domain does not result in a loss of regularity of the solution. Let us first assume that the unknown boundary data are smooth and that $\Omega$
is a convex polygonal. If only the right-hand side has low regularity, $f \in H^{-1}
(\Omega)$ we can decompose the solution into a singular part $u_s \in H^1_0(\Omega)$, solution to
\begin{equation}\label{u_s}
    -\Delta u_s + Pu_s = f  \texttt{ in $\Omega$},
\end{equation}
and $u_r \in H^r(\Omega)$ with $r\ge 2$ such that $u_r = q - u_s$ in $\omega$, and 
\begin{equation}
    -\Delta u_r +Pu_r= 0  \texttt{ in $\Omega$},
\end{equation}
Clearly in this situation we can solve (\ref{u_s}) using the standard finite element method with piece-wise affine elements, leading to the error estimate
\begin{equation}\label{boundus}
    ||u_s-u_{h,s}||_{L^2(\Omega)} \le Ch||f||_{H^{-1}(\Omega)}.
\end{equation}
Since $u_r$ is smooth we can apply the primal-dual stabilised method (see Chapter 3), with elements of polynomial degree less than or equal to $p$, to obtain the optimal approximation $u_{h,r}$ satisfying
\begin{equation*}
    ||u_r - u_{h,r}||_{L^2(B)} \le C(P) h^{\kappa (t+1)}||u_r||_{H^{t+1}(\Omega)},
\end{equation*}
where $t = \min\{p, r-1\}$, $B\subset\subset \Omega$ and $\kappa \in (0,1)$ a coefficient depending on the distance from the boundary of $B$ to the boundary of $\Omega$. As this distance goes to zero, so does $\kappa$.

It follows that in the special case where the only cause of poor regularity is the right-hand side $f$, we can still obtain an error estimate for $u_h = u_{h,s} + u_{h,r}$,
\begin{equation*}
    ||u-u_h||_{L^2(B)} \lesssim h||f||_{H^{-1}(\Omega)} + h^{\kappa (t+1)}||u_r||_{H^{t+1}(\Omega)} .
\end{equation*}

Hence for unperturbed data it is possible to achieve up to first order convergence in spite of the singular right-hand side. Of course these arguments do not work for the second and third points above. If the boundary of $\Omega$ has a physical geometric singularity then the estimate (\ref{boundus}) fails due to lack of stability of the adjoint equation and if some boundary data has insufficient smoothness then $u_r\in H^r(\Omega)$ cannot hold with $r \ge 2$. The objective of the present work is to propose an error analysis that is valid for all sources of low regularity. We will follow an optimization approach by introducing a stabilized Lagrange functional $L_h(u_h,z_h)$ at the discrete level, namely,
\begin{equation*}
\begin{aligned}
    L_h(u_h,z_h) := &\frac{1}{2}||u_h-q||^2_\omega +\textit{Primal stabilizer}-\textit{Dual stabilizer}\\
    &+a(u_h,z_h)-<f,z_h>_\Omega.
    \end{aligned}
\end{equation*}
The stabilization method combines tools known from finite element methods to ensure a discrete inf-sup condition and  Tikhonov type regularization, see, for example \cite[Chapter~5]{engl1996regularization}. Similar type of approaches have been used in \cites{burman2019unique,burman2020stabilized}. The main differences between this paper and these references are that we build up a parameterized approach adaptive to the prior knowledge of the smoothness of the exact solution $u$. The method parameters are chosen differently to show the convergence for $u\in H^1(\Omega)$ or $u\in H^s(\Omega)$, with $s > 1$. The main results are that for $u\in H^1(\Omega)$, the approximations of our scheme converges weakly under the $H^1$-norm, with error estimates on residual quantities,\\
\begin{equation*}
    ||u_h-q||_{L^2(\omega)}\le C(P)h^{\alpha}||u||_{H^1(\Omega)},
\end{equation*}
and
\begin{equation*}
    ||(-\Delta+P)u_h-f||_{H^{-2}(\Omega)}\le C(P)(h+h^{\frac{\tau}{2}})||u||_{H^1(\Omega)},
\end{equation*}
where $0<\alpha\le 1, 0<\tau\le 2$ are parameters we will use in the scheme. Moreover, we will use the error estimate of the residual norm to show convergence in the $L^2$-norm, namely,
\begin{equation*}
    ||u_h-u||_{L^2(B)} \le C(P)h^\kappa||u||_{H^1(\Omega)}.
\end{equation*}
Moreover, for $u\in H^s(\Omega)$ with $s>1$ ($s$ may be fractional), we will show that our scheme has a convergence matching the approximation order of the finite element space and the conditional stability estimate, namely,
\begin{equation*}
    ||u-u_h||_{H^1(B)}\le C(P)h^{\kappa (s-1)}||u||_{H^s(\Omega)}.
\end{equation*}

\section{Continuous stability estimate}
\label{sec:continuous}
Our stabilized finite element method relies on several conditional stability estimates at the continuous level. The layout of necessary estimates follows \cite{burman2019unique} for Helmholtz equations. We will state similar estimates for the Schrödinger equations. These estimates are based on the following well known Carleman estimate:
\begin{lemma}[Carleman estimate]\thlabel{car}
    let $\rho\in C^3(\Omega)$ and $K\subset\Omega$ be a compact subset and contains no critical point of $\rho$. Let $\alpha>0$ and $\phi=e^{\alpha\rho}$. Then for some large $\alpha$ there exists $\tau_0\ge 1$ and $C>0$ depending on $\alpha$ and $\rho$, such that for all $\tau\ge \tau_0$ and $w\in C^2_0(K)$,
    \begin{equation*}
        \int_Ke^{2\tau\phi}(\tau|\nabla w|^2+\tau^3|w|^2)dx\le C\int_Ke^{2\tau\phi}|\Delta w|^2dx.
    \end{equation*}
\end{lemma}
\begin{proof}
    See \cite[Corollary 2.3]{nechita2020unique}.
    
\end{proof}

Now we derive four Hölder type stability estimates with optimized parameters. The first one is the estimate for $H^1$-norm in the $H^2$ space. The second one is the estimate for $H^1$-norm in the $H^1$-space, and the last one is for $L^2$-norm in the $L^2$ space. The latter three estimates can be used in showing convergence of the stabilized finite element method. Their proof will be based on the first estimate.

\begin{co}[Three-ball inequality]\thlabel{tbi}
    Let $x_0\in \Omega$ and $0<r_1<r_2<d(x_0,\partial\Omega)$. Define $B_j=B(x_0,r_j)$, $j=1,2$. Then there exists $\kappa\in (0,1)$ depending only on $\Omega$, we have that $\forall \epsilon\in(0,1)$, there exists $C>0$ depending on $\epsilon$ and $\Omega$, such that for $u\in H^2(\Omega)$ it holds that:
    $$||u||_{H^1(B_2)}\le \exp(C\kappa^{1-\epsilon}(||P||^{\frac{2}{3}}_{L^\infty(\Omega)}+1))(||u||_{H^1(B_1)}+||-\Delta u+Pu||_{L^2(\Omega)})^\kappa ||u||_{H^1(\Omega)}^{1-\kappa}.$$
\end{co}

\begin{proof}
    See Appendix \ref{sec:prooftbi}
    
\end{proof}
\begin{remark}
    Note that for simplicity, we only proved \thref{tbi} for this specific geometry, say, $B_1\subset B_2\subset\Omega$, but it can be shown to hold for a larger class of geometric settings. Let such a geometry be $\omega\subset B\subset\Omega$.
\end{remark}
\begin{co}[Unique continuation property]\thlabel{uniqueness}
    Assume problem (\ref{pbl}) has a solution, and the conditional stability estimate in \thref{tbi} holds for any $B$ such that $\omega\subset B\subset \Omega$. Then the solution is unique.
\end{co}
\begin{proof}
    Let $u$, $v$ be two different solutions to (\ref{pbl}), consider the difference $w=u-v$, we have
    \begin{equation*}
    \left\{\begin{array}{rcc}
         -\Delta w+Pw& =0 & \texttt{in $\Omega$} \\
         u&  =0 & \texttt{in $\omega$}
    \end{array}.
    \right.
\end{equation*}

By the estimate in \thref{tbi}, the right-hand side vanishes, and thus $w=0$ in $B_2$. By iterating the argument $w=0$ in $\Omega$. The conclusion follows.

\end{proof}
Note also that in \thref{tbi}, $u$ is required to be $H^2(\Omega)$ because $||-\Delta u+Pu||_{L^2(\Omega)}$ is used on the right-hand side. We want to generalize the estimate to all functions in $H^1(\Omega)$.
\begin{co}\thlabel{co2}
    Let $B_1\subset B_2\subset \Omega$ be defined as in \thref{tbi}. Let $u\in H^1(\Omega)$. Then the following estimate holds
    \begin{equation*}
    \begin{aligned}
        ||u||_{H^1(B_2)}\le & C(P)(||u||_{L^2(B_1)}+||-\Delta u+Pu||_{H^{-1}(\Omega)})^\kappa\\
        &\cdot(||u||_{L^2(\Omega)}+||-\Delta u+Pu||_{H^{-1}(\Omega)})^{1-\kappa}
        \end{aligned}
    \end{equation*}
Furthermore, if $P\ge 0$ on $\Omega$, then $C(P) = \exp\left(C\kappa^{1-\epsilon}(||P||^{\frac{2}{3}}_{L^\infty(\Omega)}+1)\right)$.
\end{co}
\begin{proof}
    See Appendix \ref{sec:proofco2}.
    
\end{proof}

In \thref{co2} we derive a conditional stability estimate for $u\in H^1(\Omega)$. Now we give another bound for $||u||_{L^2(B_2)}$, $\forall u\in H^1(\Omega)$.
\begin{co}\thlabel{L2error}
    Let the geometric setting be the same as in \thref{tbi}. Then there exists $\kappa\in (0,1)$ depending only on $\Omega$, we have that $\forall \epsilon\in(0,1)$, there exists $C>0$ depending on $\epsilon$ and $\Omega$, such that $\forall u\in H^1(\Omega)$ it holds that:
    $$||u||_{L^2(B_2)}\le C(P)\left(||u||_{L^2(B_1)}+||-\Delta u+Pu||_{H^{-1}(\Omega)}\right)^\kappa ||u||_{L^2(\Omega)}^{1-\kappa},$$
    where $C(P) = \exp\left(C\kappa^{1-\epsilon}(||P||^{\frac{2}{3}}_{L^\infty(\Omega)}+1)\right)$.
\end{co}
\begin{proof}
    See Appendix \ref{sec:proofL2error}.
    
\end{proof}
Note that by shifting the $H^1(\Omega)$ norm to the $L^2(\Omega)$ norm, we get an exact constant depending on $P$ but win nothing in the order of $\kappa$. However, a recent result \cite[Theorem 2.2]{monsuur2024ultra} on the conditional stability estimate shows that \thref{L2error} is not optimal in terms of $||-\Delta u+Pu||_{H^{-1}(\Omega)}$. In fact, we have
\begin{co}[$L^2$ error estimate under the ultra-weak norm]\thlabel{ultraweak}
    Let $B_1\subset B_2\subset \Omega$ be defined as in \thref{tbi}. Let $u\in L^2(\Omega)$. Then the following estimate holds
    \begin{equation*}
    \begin{aligned}
        ||u||_{L^2(B_2)}\le & C(P)(||u||_{L^2(B_1)}+||-\Delta u+Pu||_{H^{-2}(\Omega)})^\kappa\\
        &\cdot(||u||_{L^2(\Omega)}+||-\Delta u+Pu||_{H^{-2}(\Omega)})^{1-\kappa}
        \end{aligned}
    \end{equation*}
    Furthermore, if $P\ge 0$ on $\Omega$, then $C(P) = \exp\left(C\kappa^{1-\epsilon}(||P||^{\frac{2}{3}}_{L^\infty(\Omega)}+1)\right)$.
\end{co}
\begin{proof}
    See Appendix \ref{sec:proofco2}.
    
\end{proof}
\section{Stabilized finite element method}
\label{sec:finite element}
In this section we propose a stabilized $H^1$-conforming finite element method for problem (\ref{pbl}), with the a priori knowledge that $u\in H^s(\Omega)$ for some $s\ge 1$. We will first form an equivalent optimization problem, and then stabilize the optimization problem at the discrete level. Then we show that the method gives a convergence rate up to the conditional stability. We will also discuss the stability of the discrete system and give precision thresholds.\\

\textbf{Lagrange optimization problem.} In this part, we derive the equivalent optimization problem to problem (\ref{pbl}) and give some useful propositions that we will use repeatedly. From now on, we will make some simplifications in the notation. Let $(\cdot,\cdot)_U, \langle\cdot,\cdot\rangle_U$ denote the inner products $L^2$ and $H^1$ in some domain $U$, respectively. Let $\mathcal{L}$ be the elliptic operator $-\Delta+P$, and $<\cdot,\cdot>_U$ be the duality pairing $H^{-1}(U)-H_0^1(U)$.

\begin{prop}
    Let problem (\ref{pbl}) have a unique solution, then $u\in H^1(\Omega)$ solves problem (\ref{pbl}) iff $u$ solves
    \begin{equation}\label{optpbl}
    \begin{array}{cccc}
         \underset{u\in H^1(\Omega)}{\min}\frac{1}{2}||u-q||^2_\omega &\textit{subject to}& a(u,w) = <f,w>_\Omega , & \forall w\in H^1_0(\Omega) 
    \end{array}
\end{equation}
\end{prop}

Now we construct the Lagrange functional for the optimization problem (\ref{optpbl}) with the PDE constraint. Let
\begin{equation*}
    L(u,z) := \frac{1}{2}||u-q ||^2_\omega + a(u,z) - <f, z>_\Omega, \;\;\; (u,z)\in H^1(\Omega)\times H^1_0(\Omega),
\end{equation*}
where $z\in H^1_0(\Omega)$ acts as a Lagrange multiplier. The corresponding Euler-Lagrange equation reads:
\begin{equation}\label{conel}
    \left\{\begin{array}{l}
         0=\partial_{u}Lv=(u- q,v)_\omega +a(v,z) \\
         0=\partial_{z}Lw=a(u,w)-< f,w>_\Omega 
    \end{array}\right. \;\; \forall (v,w)\in H^1(\Omega)\times H^1_0(\Omega).
\end{equation}
The continuous problem is clearly ill-posed because the equivalent PDE problem is ill-posed. Instead of regularizing the problem at this stage, we will transfer the functional optimization problem to the discrete level and then discuss the stabilization method. \\

\textbf{Finite element settings.} From now on we will consider for simplicity that $\Omega\subset\mathbb{R}^d$, $d\in \{2,3\}$, is a polygonal/polyhedral domain. Let $\{\mathcal{T}_h\}_{h>0}$ be a quasi-uniform family of shape-regular meshes covering $\Omega$, each triangulation $\mathcal{T}_h$ containing elements $K$ with maximum diameter $h$. Let $p\ge1$ denote the degree of polynomial approximation and let $\mathbb{P}_p(K)$ be the space of polynomials of degree at most $p$ defined in $K$. Let the $H^1$-conforming finite element space
$$ V^p_h :=\{v_h\in H^1(\Omega):v_h|_K\in \mathbb{P}_p(K), \forall K \in \mathcal{T}_h \}$$
and its subspace with homogeneous boundary conditions
$$W_h^p := V^p_h \cap H^1_0(\Omega).$$
Moreover, define the space of piecewise constant functions, $V_h^0 := \{v_h\in L^2(\Omega):v_h|_K= C_K\in \mathbb{R}, \forall K \in \mathcal{T}_h \}$.\\

The corresponding discrete problem will be the restriction of $(u,z)$ on $V^p_h\times W^p_h$. One can consider the Lagrange functional
\begin{equation*}
    L^{singular}_h(u_h,z_h) := \frac{1}{2}h^{-2\alpha}||u_h-q||^2_\omega + a(u_h,z_h) - <f, z_h>_\Omega, \;\;\; (u_h,z_h)\in V^p_h\times W^p_h,
\end{equation*}
where $h^{-2\alpha}$ is the mesh-specific weighting term between interior difference $u_h-q$ and the PDE constraint. We will specify the choice of $\alpha$ later.\\

Saddle points of the above Lagrange functional may not be unique and the corresponding linear system may be singular. Thus we introduce primal and dual stabilizers, which are bi-linear forms on $V^p_h$ and $W^p_h$, $s_h(\cdot,\cdot)$ and $s_h^*(\cdot,\cdot)$, and a term $G(\cdot,\cdot)$ to keep weak consistency. $G(\cdot,\cdot)$ is designed so that $s_h(u,\cdot) = G(f,\cdot)$. This leads to the stabilized Lagrange functional:
\begin{equation*}
\begin{aligned}
    L_h(u_h,z_h) := &\frac{1}{2}h^{-2\alpha}||u_h-q||^2_\omega +\frac{1}{2}s_h(u_h,u_h)\\
    &-\frac{1}{2}s_h^*(z_h,z_h)+a(u_h,z_h)-<f,z_h>_\Omega+G(f,u_h).
    \end{aligned}
\end{equation*}
The corresponding Euler-Lagrange equations read: for any $(v_h,w_h)\in V_h^p\times W^p_h$,
\begin{equation*}
    \left\{\begin{array}{l}
      0=\partial_{u_h}L_hv_h=h^{-2\alpha}(u_h-q,v_h)_\omega + s_h(u_h,v_h)+a(v_h,z_h)+G(f,v_h) \\
         0=\partial_{z_h}L_hw_h=a(u_h,w_h)-s_h^*(z_h,w_h)-<f,w_h>_\Omega, 
    \end{array}\right.
\end{equation*}
and we arrive at the stabilized finite element method: find $(u_h,z_h)\in V^p_h\times W^p_h$, such that $\forall (v_h,w_h)\in V^p_h\times W^p_h$,
\begin{equation}\label{sys}
    \left\{\begin{array}{rl}
        h^{-2\alpha}(u_h,v_h)_\omega+s_h(u_h,v_h)+a(v_h,z_h) &= h^{-2\alpha}(q,v_h)_\omega+G(f,v_h)  \\
         a(u_h,w_h)-s_h^*(z_h,w_h)&=<f,w_h>_\Omega 
    \end{array}\right..
\end{equation}\\

\textbf{Stabilization.} Now let us introduce our choice of stabilizers. Let $\mathcal{L}_h$ be the discrete Schrödinger operator, namely, $\forall v_h\in V^p_h$,
\begin{equation*}
    \mathcal{L}_hv_h := \sum_K(-\Delta +P)v_h|_K.
\end{equation*}
For a given mesh $\mathcal{T}_h$, we consider the set of all interior edges (relative to $\Omega$) / faces of elements, denoted by $\mathcal{F}_i$, and the jump function on each edge/face $F$ by $\llbracket\cdot\rrbracket_F$, omitting its subscript whenever there is no confusion. For $v_h\in V^p_h$, we denote the jump function of its normal gradient across $F\in \mathcal{F}_i$ by 
\begin{equation*}
\llbracket\nabla v_h\cdot n\rrbracket_F :=\nabla v_h\cdot n_1|_{K_1} + \nabla v_h\cdot n_2|_{K_2},
\end{equation*}
with $K_1,K_2\in\mathcal{T}$ being two elements such that $K_1\cap K_2=F$, and $n_1,n_2$ the outward pointing normals with respect to $K_1$ and $K_2$. Note that for $v\in H^{\frac32+\epsilon}(\Omega)$, $\llbracket\nabla v\cdot n\rrbracket_F =0$ is well defined. $\forall v_h \in V^p_h$, define 
\begin{equation*}
    \mathcal{J}_h(u_h,v_h) = \sum_F\int_Fh\llbracket\nabla u_h\cdot n\rrbracket\cdot\llbracket\nabla v_h\cdot n\rrbracket_FdS.
\end{equation*}

Now define $s_h:V^p_h\times V_h^p\to \mathbb{R}$ by 
\begin{equation}\label{primstab}
    s_h(u_h,v_h):= \mathcal{J}_h(u_h,v_h)+ (h\mathcal{L}_hu_h,h\mathcal{L}_hv_h)_\Omega+h^{2(s-1)}\langle u_h,v_h\rangle_\Omega,
\end{equation}
$s^*_\eta(z_h,w_h):W^p_h\times W_h^p\to\mathbb{R}$ by
\begin{equation}\label{dualstabeta}
    s^*_\eta(z_h,w_h):= h^{2\eta}\mathcal{J}_h(z_h,w_h)+h^{2\eta}\int_{\partial \Omega}h\partial_nz_h\partial_nw_hdS+h^{2\eta}( h\mathcal{L}_hz_h,h\mathcal{L}_hw_h)_\Omega,
\end{equation}
and $s^*_h$ by
\begin{equation}\label{dualstab}
    s_h^*(z_h,w_h):=s^*_\eta(z_h, w_h)+h^\tau\langle z_h, w_h\rangle_\Omega,
\end{equation}
where $h^{2\eta}$ and $h^\tau$ is the weights with respect to $s_\eta^*$ and $\langle z_h, w_h\rangle_\Omega$ in $s^*_h$. 
\begin{remark}
    The intuition for choosing $s^*$ in the form of (\ref{dualstab}) instead of using the $H^1$-inner product is that we would like the stabilizer to provide the necessary stability and vanish with mesh refinement. This is essential in order to prove error estimate for solutions with low regularity. See \thref{cvg}.
\end{remark}
Now we define the consistency term $G(f,v_h)$. Let $f_h$ be the projection from $H^{-1}(\Omega)$ to $W^p_h$ such that for any $w_h\in W^p_h$,
\begin{equation*}
    <f, w_h>_\Omega = (f_h, w_h)_\Omega.
\end{equation*}
Specifically, if $f\in L^2(\Omega)$, we have
$
    (f_h,f_h)_\Omega = (f, f_h)_\Omega \le ||f||_\Omega ||f_h||_\Omega,
$
and hence the standard stability of the $L^2$-projection holds
$
    ||f_h||_\Omega \le ||f||_\Omega.
$

Define $G:H^{-1}(\Omega)\times V^p_h\to \mathbb{R}$ by
\begin{equation}\label{G}
    G(f,v_h) = h^2(f_h, \mathcal{L}_hv_h)_\Omega.
\end{equation}
\\

\textbf{Inverse and trace inequalities.} Now we give some useful inequalities and approximation results (see, for example, \cite{di2011mathematical}).

(1) Continuous trace inequality.
\begin{equation}\label{tc1}
    ||v||_{\partial K}\le C(h^{-\frac{1}{2}}||v||_K+h^{\frac{1}{2}}||\nabla v||_K), \;\; \forall v\in H^1(K).
\end{equation}

(2) Discrete trace inequality.
\begin{equation}\label{tc2}
    ||\nabla v_h \cdot n||_{\partial K}\le Ch^{-\frac{1}{2}}||\nabla v_h||_K, \;\; \forall v_h\in \mathbb{P}_p(K).
\end{equation}

(3) Inverse inequality.
\begin{equation}\label{inv}
    ||\nabla v_h||_K\le Ch^{-1}||v_h||_K, \;\; \forall v_h\in\mathbb{P}_p(K).
\end{equation}\\

\textbf{Interpolation.} We now give interpolants we will use repeatedly. 

(1) The Scott-Zhang operator $\pi_{sz}: H^{\frac{1}{2}}(\Omega) \to V_h^p$. It keeps homogeneous boundary condition and has the optimal approximation property, see \cite{ciarlet2013analysis}:
$$ (\sum_{K\in\mathcal{T}_h}||v-\pi_{sz}v||^2_{H^m(K)})^\frac{1}{2} \le Ch^{s-m}|v|_{H^{s}(\Omega)}, $$
whenever $v\in H^{s}(\Omega)$ with $\frac{1}{2}<s\le p+1$ and $m\in \{0:[s]\}$.\\

(2) $L^2$-projection $\pi_0: L^2 \to V_h^0$. It has local optimal approximation property, see \cite[Theorem 18.16]{ern2021finite}:
$$||v-\pi_0v||_K \le Ch^{s}|v|_{H^{s}(K)},$$
whenever $K\in \mathcal{T}_h$, $v\in H^{s}(K)$ with $0\le s \le 1$. Note that $V_h^0$ is not $H^1$-conforming.\\

(3) $H^1$-conforming $L^2$-projection $\pi_h: L^2 \to V_h^p$ with $p\ge 1$. It is the best approximation in the $L^2$-norm, and thus it has optimal approximation property in $L^2$-norm, see \cite[Corollary 22.9]{ern2021finite}
\begin{equation*}
    ||v-\pi_hv||_\Omega \le C h^s|v|_{H^{s}(\Omega)},
\end{equation*}
whenever $v\in H^s(\Omega)$ with $0\le s \le p+1$.\\

\textbf{Properties of the stabilized system.} We first show that $\mathcal{J}_h(v_h,v_h)$ and $h^2(\mathcal{L}_hv_h,\mathcal{L}_hv_h)_\Omega- G(f_h,v_h)$ are weakly consistent.
\begin{prop}\thlabel{weakcons}
    Let $\pi_{sz}$ be the Scott-Zhang operator and $v\in H^s(\Omega)$ for $s\ge 1$. Then there is a constant $C$ such that
    \begin{equation*}
         \mathcal{J}_h(\pi_{sz}v,\pi_{sz}v)\le Ch^{2(s-1)}||v||^2_{H^{s}(\Omega)}.
    \end{equation*}
    Moreover, if $s<2$ we have
    \begin{equation*}
        ||h\mathcal{L}_h(\pi_{sz}v)||_\Omega \le C(1+||P||_{L^\infty})h^{s-1}||v||_{H^s(\Omega)},
    \end{equation*}
    and if $s\ge 2$, we have
    \begin{equation*}
        ||h\mathcal{L}_h(\pi_{sz}v-v)||_\Omega \le C(1+h^2||P||_{L^\infty})h^{s-1}||v||_{H^s(\Omega)}.
    \end{equation*}
\end{prop}
\begin{proof}
\textbf{Step1}. We give the estimate for the jump of gradients. Let $\pi_h^n:(L^2(\Omega))^n \to (V_h^p)^n$ be defined by
\begin{equation*}
    \pi_h^n\mathbf{u} = (\pi_hu_1,\pi_hu_2,...,\pi_hu_n)^t,
\end{equation*}
where $\mathbf{u}= (u_1,u_2,...,u_n)^t$. Since the range of $\pi_h$ is in $H^1(\Omega)$, $\llbracket \pi_h w\rrbracket = 0$ for all $w\in L^2(\Omega)$. This gives
    \begin{equation*}
    \begin{aligned}
         \mathcal{J}_h(\pi_{sz}v,\pi_{sz}v) 
         &= \underset{F\in \mathcal{F}_i}{\sum}\int_Fh\llbracket(\nabla \pi_{sz}v-\pi_h^n\nabla \pi_{sz}v)\cdot n\rrbracket_F\llbracket (\nabla \pi_{sz}v-\pi_h^n\nabla \pi_{sz}v)\cdot n\rrbracket_FdS\\
         &\le C\sum_K||\nabla \pi_{sz}v-\pi_h^n\nabla \pi_{sz}v_h||^2_K\\
         &\le C(||\nabla (\pi_{sz}v - v)||^2_\Omega + ||\nabla v - \pi_h^n \nabla v)||^2_\Omega + ||\pi_h^n \nabla (v -\pi_{sz}v)||^2_\Omega) \\
         &\le Ch^{2(s-1)}||v||^2_{H^s(\Omega)}
    \end{aligned}
    \end{equation*}
where the first inequality is from (\ref{tc1}), (\ref{inv}) and the third inequality is from the $H^1$ optimal approximation property of $\pi_{sz}$ and the $L^2$ optimal approximation property of $\pi^n_h$.\\

\textbf{Step 2.} Now we only need to give the estimate for $||h\mathcal{L}_h(\pi_{sz}v)||_\Omega$ when $1\le s < 2$, since when $s\ge 2$ the estimate is given immediately by the optimal approximation property of $\pi_{sz}$. Let $\pi_0^n:(L^2(\Omega))^n \to (V_h^0)^n$ be defined similarly as $\pi_h^n$. We have $\nabla\cdot\pi_0^n\nabla i_hv = 0$ on each $K\in\mathcal{T}_h$ since $\pi_0$ maps a function to piece-wise constants. Let $\Delta_h$ be the element-wise Laplacian, we have
\begin{equation*}
    \begin{aligned}
        ||h\Delta_h \pi_{sz}v||_\Omega &= (\sum_K||h\nabla \cdot \nabla \pi_{sz}v||^2_K)^\frac{1}{2} \\
        & = (\sum_K||h\nabla\cdot (\nabla \pi_{sz}v-\pi_0^n \nabla \pi_{sz}v)||_K^2)^\frac{1}{2} \\
        &\le C(\sum_{K}||\nabla \pi_{sz}v-\pi_0^n \nabla \pi_{sz}v||^2_K)^\frac{1}{2} \\
        & \le C\left(\sum_K(||\nabla (\pi_{sz}v-v)||^2_K+||\nabla v - \pi_0^n \nabla v||^2_K +||\pi_0^n\nabla(v-\pi_{sz}v)||^2_K)\right)^\frac{1}{2} \\
        & \le Ch^{s-1} ||v||_{H^s(\Omega)},
    \end{aligned}
\end{equation*}
where the first inequality is from (\ref{inv}), the third inequality is from optimal approximation property of $\pi_{sz}$ and local optimal approximation property of $\pi_0$, and the last equality is from the summability of $H^1$ functions. The estimate for $||h\mathcal{L}_h(\pi_{sz}v)||_\Omega$ is then obtained by adding the potential term.

\end{proof}

Now we show that the discrete problem is well-posed. Let
\begin{equation}\label{stabnorm}
    |||(v_h,w_h)|||^2=h^{-2\alpha}(v_h,v_h)_\omega+s_h(v_h,v_h)+s_h^*(w_h,w_h).
\end{equation}
This is a norm since we have (c.f.(\ref{primstab}),(\ref{dualstab}))
\begin{equation*}
    h^{2(s-1)}\langle v_h,v_h\rangle_\Omega \le s_h(v_h,v_h),
\end{equation*}
and
\begin{equation*}
    h^\tau\langle w_h,w_h\rangle_\Omega \le s_h^*(w_h,w_h).
\end{equation*}
Note that also in the absence of Tikhonov regularization \eqref{stabnorm} is a norm thanks to unique continuation. It is however not strong enough to allow for error estimates.

Now define the discrete bilinear form
\begin{equation}\label{A}
\begin{aligned}
    A[(u_h,z_h),(v_h,w_h)]:=& h^{-2\alpha}(u_h,v_h)_\omega+s_h(u_h,v_h)+ a(v_h,z_h)\\
    &+a(u_h,w_h)-s_h^*(z_h,w_h).
    \end{aligned}
\end{equation}
\begin{prop}\thlabel{BNBA}
    The linear system defined by (\ref{sys}) is well-posed. In fact, the following inf-sup condition holds
    \begin{equation*}
    \underset{(v_h,w_h)\in V_h^p\times W_h^p}{\sup}\frac{A_h[(u_h,z_h),(v_h,w_h)]}{|||(v_h,w_h)|||}\ge |||(u_h,z_h)|||.
    \end{equation*}
\end{prop}
\begin{proof}
Let $v_h=u_h$, $w_h=-z_h$ we have
\begin{equation*}
    \begin{aligned}
        A_h[(u_h,z_h),(u_h,-z_h)]&=h^{-2\alpha}||u_h||^2_\omega+s_h(u_h,u_h)+s_h^*(z_h,z_h)\\
        &=|||(u_h,z_h)|||^2.
    \end{aligned}
\end{equation*}
\end{proof}
\thref{BNBA} gives the stability of the linear system, now we give an upper bound of the stabilizers.
\begin{prop}\thlabel{boundA}
    Let $[(u_h,z_h),(v_h,w_h)]\in [V_h^p\times W_h^p]^2$, $0\le \alpha\le1$ and $\eta,\tau$ non-negative. We have the following bound
    \begin{equation*}
        s_h(u_h,v_h)+s^*_h(z_h,w_h)\le C(P)\left(||u_h||_{H^1(\Omega)}||v_h||_{H^1(\Omega)}+||z_h||_{H^1(\Omega)}||w_h||_{H^1(\Omega)}\right),
    \end{equation*}
    where $C(P) = C(1+h^2||P||^2_{L^\infty}).$
\end{prop}
\begin{proof}
    Consider each component of $s_h$. First we have by (\ref{tc1}) and (\ref{inv})
    \begin{equation}\label{boundofjump}
    \begin{aligned}
        \mathcal{J}_h(u_h,v_h)&\le C\sum_K(h||u_h||_{H^2(K)}+||u_h||_{H^1(K)})(h||v_h||_{H^2(K)}+||v_h||_{H^1(K)})\\
        &\le C||u_h||_{H^1(\Omega)}||v_h||_{H^1(\Omega)}\\
    \end{aligned}
    \end{equation}
    Similarly for the second component we also have 
    \begin{equation}\label{boundofres}
        h^2(\mathcal{L}_hu_h,\mathcal{L}_hv_h)_\Omega\le C(1+h^2||P||^2_{L^\infty})||u_h||_{H^1(\Omega)}||v_h||_{H^1(\Omega)}.
    \end{equation}
    Since $W_h^p\subset V_h^p$, we have by (\ref{boundofjump}) and (\ref{boundofres})
    \begin{equation}\label{s*bound1}
        h^{2\eta}\mathcal{J}_h(z_h,w_h)\le Ch^{2\eta}||z_h||_{H^1(\Omega)}||w_h||_{H^1(\Omega)},
    \end{equation}
    and 
    \begin{equation}\label{s*bound2}
        h^{2\eta}(h\mathcal{L}_hz_h,h\mathcal{L}_hw_h)_\Omega\le C(1+h^2||P||^2_{L^\infty})h^{2\eta}||z_h||_{H^1(\Omega)}||w_h||_{H^1(\Omega)}.
    \end{equation}
    Again, by (\ref{tc2}), we have
    \begin{equation}\label{s*bound3}
        h^{2\eta}\int_{\partial \Omega}h\partial_nz_h\partial_nw_hdS \le Ch^{2\eta}||z_h||_{H^1(\Omega)}||w_h||_{H^1(\Omega)},
    \end{equation}
    and by Cauchy–Schwarz inequality,
    \begin{equation}\label{s*bound4}
        h^\tau(\nabla z_h,\nabla w_h)_\Omega \le h^\tau||z_h||_{H^1(\Omega)} ||w_h||_{H^1(\Omega)}.
    \end{equation}
    Combining (\ref{boundofjump}), (\ref{boundofres}), (\ref{s*bound1}), (\ref{s*bound2}), (\ref{s*bound3}) and (\ref{s*bound4}), we have the desired estimate.
    
\end{proof}

\section{Stability and error analysis}
In this section we discuss the error and stability estimate for the stabilized scheme (\ref{sys}). This will be separated into two cases: $u\in H^s(\Omega)$ for $1\le s<2$ and $s\ge 2$.

\subsection{Convergence analysis for solutions with low regularity}\label{lowreg}
In this part we assume $u\in H^s(\Omega)$ with $1\le s <2$ and $f\in H^{-1}(\Omega)$. Then we prove the solution of (\ref{sys}) converges in this case. Our aim is to show that the discrete solution satisfies an error estimate that is optimal up to the conditional stability estimate. Specifically, we derive optimal convergence rates in the following sense:

\begin{theorem}[Optimal $L^2$ error estimate]\thlabel{ultraweakconv}
    Let $u\in H^s(\Omega)$, $\alpha=1$, $\tau=2$ and $\eta=0$. Let $B$ be a region satisfying estimates in \thref{co2} and \thref{ultraweak}, then the stabilized finite element approximation $u_h$ converges to $u$ under $L^2$ norm optimally up to the conditional stability estimate \thref{ultraweak}, and we have the estimate
    \begin{equation*}
        ||u-u_h||_{B}\le C(P)h^{\kappa s}||u||_{H^s(\Omega)}
    \end{equation*}
\end{theorem}

\begin{theorem}[Optimal $H^1$ error estimate]\thlabel{prime12}
    Let $u\in H^s(\Omega)$ and $(u_h,z_h)$ be the solutions to system (\ref{sys}). Let $0\le \alpha\le 1, \eta\tau = 0$ and both non-negative. Let $B$ be a region satisfying estimate in \thref{co2} and \thref{ultraweak}; $\kappa,\epsilon$ be the parameters therein. Then
    $$||u-u_h||_{H^1(B)}\le C(P)h^{\kappa (s-1)}||u||_{H^s(\Omega)}$$
\end{theorem}
Notice that we have to specify the choice of parameters in \thref{ultraweakconv} to obtain the optimal $L^2$ convergence, while for $H^1$ convergence in \thref{prime12} we  have some flexibility in the choice of parameters. In the remainder of this subsection, we will prove these theorems.\\

First, we show that the stabilizing terms $s_h(\pi_{sz}u,v_h)$ and $G(f,v_h)$ are bounded by $||u||_{H^s(\Omega)}$ and $|||v_h|||$.

\begin{lemma}\thlabel{consis1}
    Let $u\in H^s(\Omega)$ with $1\le s <2$. $\pi_{sz}u$ be the Scott-Zhang interpolant of $u$ on $V^p_h$. Then there exists some constant $C>0$, such that $ \forall v_h\in V^p_h$,
    \begin{equation*}
        s_h(\pi_{sz}u,v_h) \le C(1+||P||_{L^\infty})h^{s-1}||u||_{H^s(\Omega)} \cdot |||v_h|||,
    \end{equation*}
    and
    \begin{equation*}
        G(f,v_h)\le C(1+||P||_{L^\infty})h^{s-1}||u||_{H^s(\Omega)} \cdot |||v_h|||.
    \end{equation*}
\end{lemma}
\begin{proof}
    First we have \begin{equation*}
    \begin{aligned}
    s_h(\pi_{sz}u,v_h)&=\mathcal{J}_h(\pi_{sz}u,v_h)+h^{2(s-1)}\langle\pi_{sz}u,v_h\rangle_\Omega+(h\mathcal{L}_h\pi_{sz}u,h\mathcal{L}_hv_h)_\Omega.
    \end{aligned}
\end{equation*}
by \thref{weakcons}, we have
\begin{equation*}
  \mathcal{J}_h(\pi_{sz}u,v_h) \le Ch^{s-1}||u||_{H^s(\Omega)}(\mathcal{J}_h(v_h,v_h))^\frac{1}{2}.
\end{equation*}
and
\begin{equation*}
\begin{aligned}
    (h\mathcal{L}_h\pi_{sz}u,h\mathcal{L}_hv_h)_\Omega &\le ||h\mathcal{L}_h\pi_{sz}u||_\Omega||h\mathcal{L}_hv_h||_\Omega\\
    & \le C(1+||P||_{L^\infty})h^{s-1}||u||_{H^s(\Omega)}||h\mathcal{L}_hv_h||_\Omega.\\
\end{aligned}
\end{equation*}
Now since 
\begin{equation*}
    G(f,v_h) = h^2(f_h, \mathcal{L}_hv_h)_\Omega \le ||hf_h||_\Omega||h\mathcal{L}_hv_h||_\Omega,
\end{equation*}
we only need to estimate $||hf_h||_\Omega$. Recall that by the definition of $f_h\in W_h^p$, we have $\forall w_h\in W^p_h$
\begin{equation*}
    \begin{aligned}
        h(f_h,w_h) = &ha(u, w_h)\\
        =&h\int_\Omega\nabla(u-\pi_{sz}u)w_h+ P(u-\pi_{sz}u)w_hdx\\
        &+h\int_\Omega\nabla \pi_{sz}u\cdot \nabla w_h+ P\pi_{sz}uw_hdx\\
         \le &C(1+h^2||P||_{L^\infty})h^{s-1}||u||_{H^s(\Omega)}||w_h||_\Omega \\
         &+ h\sum_F\int_F\llbracket\nabla \pi_{sz}u\cdot n\rrbracket w_hdS+(h\mathcal{L}_h\pi_{sz}u,w_h)_\Omega\\
         \le &C(1+||P||_{L^\infty})h^{s-1}||u||_{H^s(\Omega)}||w_h||_\Omega,
    \end{aligned}
\end{equation*}
where the first inequality is due to optimal approximation property of Scott-Zhang interpolant and the second inequality is obtained by \thref{weakcons} and applying discrete trace inequality (\ref{tc2}). Finally, by the definition of the triple norm, $|||v_h|||$, we obtain the desired result.

\end{proof}

Now we give the second lemma, which gives an estimate for the difference between $\pi_{sz}u$ and the discrete solution $u_h$ in the $|||\cdot|||$-norm.
\begin{lemma}\thlabel{stabnormest1}
Let $(u_h,z_h)$ be the solution to (\ref{sys}), $u\in H^s(\Omega)$, with $1\le s <2$, be the exact weak solution to the continuous problem, $0\le\alpha\le 1$ and $\eta\tau=0$. Then there exists some constant $C$, such that
\begin{equation*}
    |||(u_h-\pi_{sz}u,z_h)|||\le C(1+||P||_{L^\infty})h^{s-1}||u||_{H^s(\Omega)}.
\end{equation*}
\end{lemma}
\begin{proof}
We have from the $\inf-\sup$ condition that we only need to bound 
\begin{equation*}
    A[(u_h-\pi_{sz}u,z_h),(v_h,z_h)]\le Ch^{s-1}||u||_{H^s(\Omega)}|||(v_h,w_h)|||.
\end{equation*}
It follows that 
\begin{equation}\label{Aresidual}
\begin{aligned}
    A[(u_h-\pi_{sz}u,z_h),(v_h,w_h)]=&h^{-2\alpha}(u-\pi_{sz}u,v_h)_\omega\\
    &+a(u-\pi_{sz}u,w_h)-s(\pi_{sz}u,v_h)+G(f,v_h).
    \end{aligned}
\end{equation}
For the first term, we have
\begin{equation}\label{es1}
    h^{-2\alpha}(u-\pi_{sz}u,v_h)_\omega\le Ch^{s-\alpha}||u||_{H^s(\Omega)}h^{-\alpha}||v_h||_\omega.
\end{equation}
The second and third term are bounded according to \thref{consis1}, and that is where the $C(1+||P||_{L^\infty})$ in the estimate enters. For the first  term of the second line in (\ref{Aresidual}), if $\eta=0$ we have
\begin{equation}\label{es3}
    \begin{aligned}
    a(u-\pi_{sz}u,w_h) = &\int_\Omega \nabla(u-\pi_{sz} u)\cdot\nabla w_hdx+\int_\Omega P(u-\pi_{sz}u)w_hdx\\
    \le &\sum_F\int_F(u-\pi_{sz}u) \llbracket\nabla w_h\cdot n\rrbracket dS\\
    &+\int_{\partial\Omega}(u-\pi_{sz}u)\partial_n w_h dS+(u-\pi_{sz}u,\mathcal{L}_hw_h)_\Omega\\
    \le &C(h^{-1} \| u-\pi_{sz}u\|_\Omega + |u-\pi_{sz}u|_{H^1(\Omega)})(s^*(w_h,w_h))^\frac{1}{2}\\
    &\le Ch^{s-1}||u||_{H^s(\Omega)}(s^*(w_h,w_h))^\frac{1}{2},
    \end{aligned}
\end{equation}
where the second estimate is by (\ref{tc1}) and the third estimate is by the optimal convergence property of the Scott-Zhang interpolant. If instead, $\tau =0$, then by the optimal convergence property of Scott-Zhang interpolant again we have
\begin{equation*}
    \begin{aligned}
        a(u-\pi_{sz}u,w_h) &\le C(1+h||P||_{L^\infty})h^{s-1}||u||_{H^s(\Omega)}||w_h||_{H^1(\Omega)}\\
        &\le C(1+h||P||_{L^\infty})h^{s-1}||u||_{H^s(\Omega)}(s^*(w_h,w_h))^\frac{1}{2}
    \end{aligned}
\end{equation*}
Then, by combining the estimates for each term, we conclude that the estimate in the lemma holds.

\end{proof}
From \thref{stabnormest1} we can get the boundedness of the discrete solution $u_h$.
\begin{co}\thlabel{bdness1}
Let $u\in H^s(\Omega)$ with $1\le s <2$, then 
\begin{equation*}
    ||u_h||_{H^1(\Omega)} \le C(1+||P||_{L^\infty})||u||_{H^s(\Omega)},
\end{equation*}
and 
\begin{equation*}
    h^\frac{\tau}{2}||z_h||_\Omega \le C(1+||P||_{L^\infty})h^{s-1}||u||_{H^s(\Omega)}.
\end{equation*}
\end{co}
\begin{proof}
Since
\begin{equation*}
\begin{aligned}
    ||u-u_h||_{H^1(\Omega)}& \le ||u-\pi_{sz}u||_{H^1(\Omega)}+||\pi_{sz}u-u_h||_{H^1(\Omega)}\\
    &\le ||u||_{H^1(\Omega)}+h^{-(s-1)}|||(\pi_{sz}u-u_h,0)|||\\
    &\le C(1+||P||_{L^\infty})||u||_{H^s(\Omega)},
    \end{aligned}
\end{equation*}
we have $||u_h||_{H^1(\Omega)}\le C(1+||P||_{L^\infty})||u||_{H^s(\Omega)}$.
Similarly, we have
\begin{equation*}
    h^\frac{\tau}{2}||z_h||_\Omega \le |||(u-\pi_{sz}u,z_h)|||\le C(1+||P||_{L^\infty})h^{s-1}||u||_{H^s(\Omega)}.
\end{equation*}
\end{proof}
Now we derive the optimal convergence theorems under both the $L^2$ and the $H^1$ norm for system (\ref{sys}) when $u\in H^s(\Omega)$ with $1\le s<2$. We will first show that the residual vanishes in $H^{-2}$ and $H^{-1}$ norms respectively.

\begin{theorem}[$H^{-2}$-residual error estimate for $u\in H^s(\Omega)$]\thlabel{cvg}
Let $u\in H^s(\Omega)$, $0<\alpha\le1 $, $\eta=0$, and $\tau>0$. Then the stabilized finite element approximation $u_h$ converges weakly to $u$ and we have the following estimate:\\
(1) $||u_h-q||_\omega\le C(1+||P||_{L^\infty})(h^s+h^{s-1+\alpha})||u||_{H^s(\Omega)},$\\
(2) $||\mathcal{L}u_h-f||_{H^{-2}(\Omega)}\le C(1+||P||^2_{L^\infty})(h^s+h^{s-1+\frac{\tau}{2}})||u||_{H^s(\Omega)}$.
\end{theorem}
\begin{remark}
    \thref{cvg} provides an error estimate for residuals in the $H^{-2}$-norm when the solution $u$ only has $H^1$ regularity. We want to emphasize that this is only possible when we set $\eta=0$, which means we need the stability provided by $s^*_\eta$ defined in (\ref{dualstabeta}).
\end{remark}
\begin{proof}[Proof of \thref{cvg}]
We first prove the estimates in \thref{cvg}. By the density of $H^2_0(\Omega)$ in $L^2(\Omega)$, we only need to show that $\forall w\in H^2_0(\Omega)$,
\begin{equation*}
    (u_h-u,w)_\omega\le C(h^s+h^{s-1+\alpha})||u||_{H^s(\Omega)}||w||_\omega,
\end{equation*}
and
\begin{equation*}
    a(u_h-u,w)\le C(1+||P||^2_{L^\infty})(h^s+h^{s-1+\frac{\tau}{2}})||u||_{H^s(\Omega)}||w||_{H^2(\Omega)}.
\end{equation*}
Recall that the discrete system reads
\begin{equation*}
    \left\{\begin{array}{rl}
        h^{-2\alpha}(u_h,v_h)_\omega+s_h(u_h,v_h)+a(v_h,z_h) &= h^{-2\alpha}(q,v_h)_\omega  \\[3mm]
         a(u_h,w_h)-s_h^*(z_h,w_h)&=<f,w_h>_\Omega 
    \end{array}\right.,\;\quad\forall (v_h,w_h)\in V_h\times W_h.
\end{equation*}
First, we estimate the residual $a(u_h-u,w)$. Let $w_h$ be the Scott-Zhang interpolant of $w$ we have $a(u_h-u,w)=a(u_h-u,w-w_h)+a(u_h-u,w_h)$. For the second term, we have 
\begin{equation*}
    \begin{aligned}
    a(u_h-u,w_h) &= s^*(z_h,w_h)\\
    &\le (s_h^*(z_h,z_h))^\frac{1}{2}(s_h^*(w_h,w_h))^\frac{1}{2}\\
   &\le C(1+||P||_{L^\infty})h^{s-1}||u||_{H^s(\Omega)}\\ & \qquad \times ( \mathcal{J}_h(w_h,w_h)+\int_{\partial \Omega}h\partial_nw_h\partial_nw_hdS+h^{2}(\mathcal{L}_hw_h,\mathcal{L}_hw_h)_\Omega+ h^\tau||w_h||_{H^1(\Omega)})^\frac{1}{2}\\
    &\le C(1+||P||^2_{L^\infty})(h^s+h^{s-1+\frac{\tau}{2}})||u||_{H^s(\Omega)}||w||_{H^2(\Omega)}.
    \end{aligned}
\end{equation*}
Note that we use the fact that the Scott-Zhang operator is bounded with respect to the $H^2$-norm in the last inequality. For the first term, we have
\begin{equation*}
    \begin{aligned}
    a(u_h-u,w-w_h)\le C(1+h^2||P||_{L^\infty})||u_h-u||_{H^1(\Omega)}||w-w_h||_{H^1(\Omega)}\le Ch^s||u||_{H^s(\Omega)}||w||_{H^2(\Omega)}.
    \end{aligned}
\end{equation*}
Now for $(u_h-u,w)_\omega$, we have
\begin{equation*}
    h^{-2\alpha}(u_h-\pi_{sz}u,w)_\omega\le |||(u-\pi_{sz}u,0)|||\cdot ||h^{-\alpha} w||_\omega\le C(1+||P||_{L^\infty})h^{s-1-\alpha}||u||_{H^s(\Omega)}||w||_\omega,
\end{equation*}
and
\begin{equation*}
    (u-\pi_{sz}u,w)_\omega\le Ch^s||u||_{H^s(\Omega)}||w||_\omega,
\end{equation*}
which gives 
\begin{equation*}
    (u_h-u,w)_\omega\le C(1+||P||_{L^\infty})(h^s+h^{s-1+\alpha})||u||_{H^s(\Omega)}||w||_\omega.
\end{equation*}
Now we show the weak convergence. Since by \thref{bdness1} we know that $u_h$ has a subsequence that converges weakly in $H^1(\Omega)$. By the above estimates and the uniqueness of solutions to the continuous problem (\thref{uniqueness}), we know that the limit has to be the solution $u$. From the previous argument, we obtain that the $\{u_h\}$ can not have another weak limit. We conclude that $u_h \rightharpoonup u$ in $H^1(\Omega)$.

\end{proof}
Notice that we derived a weak convergence result from \thref{cvg} without specifying $\alpha$, and $\tau$. However, different parameters may affect the condition number of the discrete system and the convergence rate with respect to $||u_h-q||_\omega$ and $||\mathcal{L}u_h-f||_{H^{-2}(\Omega)}$ as indicated in \thref{cvg}.
\begin{theorem}\thlabel{H-2norm}
Assume that all parameters satisfy the conditions in \thref{cvg}. Denote the convergence rate of $||\mathcal{L}u_h-f||_{H^{-2}(\Omega)}$ by $\lambda =\min\{ s-1+\frac{\tau}{2}, s\}$, i.e.
$$||\mathcal{L}u_h-f||_{H^{-2}(\Omega)}\le Ch^\lambda ||u||_{H^1(\Omega)}.$$
Then, the discrete equations (\ref{sys}) has a condition number $$\mathcal{K}=\frac{C(P)}{h^{2\lambda+2}},$$
where $C(P) = C(1+h^2||P||^2_{L^\infty})$.
\end{theorem}

\begin{proof}
First, we notice that $\alpha$ does not affect the condition number in \thref{cvg} because the largest singular value is bounded by $\mathcal{J}_h$, which is $\frac{1}{h^2}$. Thus we set $\alpha=1$ to gain an optimal convergence rate for $||u_h-u||_\omega$. In this case, the largest singular value is $\frac{C(P)}{h^2}$, by the same argument as in \thref{boundA}. The smallest singular value is obtained by letting $(v_h,w_h)=(u_h,-z_h)$. by Poincaré's inequality there holds
\begin{equation*}
    A[(u_h,z_h),(u_h,-z_h)] = h^{-\alpha}(u_h,u_h)_\omega+s_h(u_h,u_h)+s^*_h(z_h,z_h)\ge  Ch^{2(s-1)+\tau}||(u_h,z_h)||_\Omega^2.
\end{equation*}
In this case $\mathcal{K}=C(P)h^{-(2\lambda+2)}$.
\end{proof}

Now we have the bound for the $H^{-2}$-norm, which will be used to give a optimal convergence rate under $L^2$-norm. We also give a bound for the residual in the $H^{-1}$-norm, which is important to obtain an optimal convergence in the $H^1$-norm.
\begin{theorem}[$H^{-1}$-residual error estimate for $u\in H^s(\Omega)$]\thlabel{H-1error12}
Let $u\in H^s(\Omega)$ with $1\le s < 2$, $0\le \alpha\le 1$. $\tau$ and $\eta$ are non-negative and $\eta\tau = 0$. Then we have the estimate\\
(1) $||u_h-q||_\omega\le C(1+||P||_{L^\infty})h^{s-1}||u||_{H^s(\Omega)},$\\
(2) $||\mathcal{L}u_h-f||_{H^{-1}(\Omega)}\le C(1+||P||_{L^\infty})h^{s-1}||u||_{H^s(\Omega)}$.
\end{theorem}
\begin{proof}
First we bound $||u_h-u||_\omega$. We have from \thref{stabnormest1}
\begin{equation}\label{interiorbound}
\begin{aligned}
    ||u_h-u||_\omega&\le ||u_h-\pi_{sz}u||_\omega +||u-\pi_{sz}u||_\omega\\
    &\lesssim h^{\alpha}|||(u_h-\pi_{sz}u,0)|||+h^s||u||_{H^s(\Omega)}\\
    &\le C(1+||P||_{L^\infty})h^{s-1+\alpha}||u||_{H^s(\Omega)}.
    \end{aligned}
\end{equation} 

Now for the residual, let $w_h$ be the Scott-Zhang interpolant of $w$, then we have
\begin{equation*}
    a(u_h-u,w) = a(u_h-u,w_h)+a(u_h-\pi_{sz}u,w-w_h)+a(\pi_{sz}u-u,w-w_h).
\end{equation*}
We bound the three terms separately. First we have
\begin{equation}\label{a(u_h-u,w_h)}
\begin{aligned}
        a(u_h-u,w_h) =s_h^*(z_h,w_h)&\le C|||(0,z_h)|||\cdot |||(0,w_h)|||\\
        &\le C(1+||P||_{L^\infty})|||(0,z_h)|||\cdot||w||_{H^1(\Omega)},
\end{aligned}
\end{equation}
due to the boundedness of $|||\cdot|||$ in \thref{boundA} and the boundedness of Scott-Zhang operator.
Then, for $a(u_h-\pi_{sz}u,w-w_h)$, there holds
\begin{equation}\label{a(u_h-pi_hu,w-w_h)}
    \begin{aligned}
        a(u_h-\pi_{sz}u,w-w_h) =& \int_\Omega \nabla(u_h- \pi_{sz}u)\cdot\nabla (w-w_h)dx\\
        &+ (P(u_h-\pi_{sz}u),w-w_h)_\Omega\\
        =&\sum_K\int_K (-\Delta+P)(u_{sz}-\pi_hu)(w-w_h)dx\\
        &+\sum_F\int_F\llbracket \nabla(u_h-\pi_{sz}u)\cdot n\rrbracket(w-w_h)dS\\
        \le &C||h\mathcal{L}_h(u_h-\pi_{sz}u)||_\Omega||h^{-1}(w-w_h)||_\Omega\\
        &+(\mathcal{J}_h(u_h-\pi_{sz}u,u_h-\pi_{sz}u))^\frac{1}{2}||w||_{H^1(\Omega)}\\
        \le &C|||(u_h-\pi_{sz}u,0)|||\cdot ||w||_{H^1(\Omega)},
    \end{aligned}
\end{equation}
and for $a(u-\pi_{sz}u,w-w_h)$,
\begin{equation}\label{a(u-pi_{sz}u,w-w_h)}
    a(u-\pi_{sz}u,w-w_h)\le C(1+h^2||P||_{L^\infty})h^{s-1}||u||_{H^s(\Omega)}||w||_{H^1(\Omega)}.
\end{equation}
By combining (\ref{interiorbound}), (\ref{a(u_h-pi_hu,w-w_h)}), (\ref{a(u_h-u,w_h)}), (\ref{a(u-pi_{sz}u,w-w_h)}) we get the desired result.

\end{proof}
Since we have obtained the residual bound under both $H^{-1}$ and $H^{-2}$ norms in \thref{cvg} and \thref{H-1error12}, we can use them to prove \thref{ultraweakconv} and \thref{prime12}.

\begin{proof}[Proof of \thref{ultraweakconv} and \thref{prime12}]
    From conditional stability result we have for $u-u_h$,
    \begin{equation*}
    ||u-u_h||_{H^1(B)}\le C(P)(||u-u_h||_\omega+||\mathcal{L}(u-u_h)||_{H^{-1}})^\kappa(||u-u_h||_\Omega+||\mathcal{L}(u-u_h)||_{H^{-1}})^{1-\kappa}.
    \end{equation*}
    and 
    \begin{equation*}
        ||u-u_h||_B\le C(P)(||u-u_h||_\omega+||\mathcal{L}(u-u_h)||_{H^{-2}})^\kappa(||u-u_h||_\Omega+||\mathcal{L}(u-u_h)||_{H^{-2}})^{1-\kappa}.
    \end{equation*}
    Note that from \thref{H-1error12}, $||u-u_h||_\omega$ and $||\mathcal{L}(u-u_h)||_{H^{-1}}$ converge optimally. Bounding
    \begin{equation*}
    \begin{aligned}
        ||u-u_h||_\Omega &\le Ch^{-(s-1)}|||(u-u_h,z_h)|||\\
        &\le C(1+||P||_{L^\infty})||u||_{H^{s}(\Omega)},
    \end{aligned}
    \end{equation*}
    we obtain that 
    $$||u-u_h||_{H^1(B)}\le C(P)h^{\kappa(s-1)}||u||_{H^{s}(\Omega)}.$$
    By repeating the procedure and applying \thref{cvg}, in place of \thref{H-1error12}, we obtain the estimate for the L2-norm on B.
    
\end{proof}

\subsection{Optimal convergence for solutions with higher regularity}
In this part we assume that $u$ enjoys a higher regularity, namely, $u\in H^s(\Omega)$ for some $s\ge 2$, and $f\in H^{s-2}(\Omega)$. In this case, we may assume $P$ has a higher regularity $P\in W^{s-2,\infty}(\Omega)$. The layout and the main result would be similar as in Section \ref{lowreg}.The distinction lies in our analysis of the system's stability under perturbations and interpolation errors. We will display the trade-off between the $L^2$ convergence rate and the stability of the system.\\

The discrete system (\ref{sys}) then reads:
\begin{equation}\label{smoothsys}
    \left\{\begin{array}{rl}
        h^{-2\alpha}(u_h,v_h)_\omega+s_h(u_h,v_h)+ a(v_h,z_h) &= h^{-2\alpha}(\tilde q,v_h)_\omega +(h\tilde f_h, h\mathcal{L}_hv_h) \\
         a(u_h,w_h)-s_h^*(z_h,w_h)&=(\tilde f_h,w_h)_\Omega 
    \end{array}\right.,
\end{equation}
where $u_h \in V_h^p$ and $w_h\in W_h^p$ are finite element spaces with $p+1\ge s$. In this session, we will consider the convergence result with perturbed data (at the discrete level), namely
\begin{equation}
    \tilde f_h = f_h+\delta f, \;\;\; \tilde q = q + \delta q,
\end{equation}
where $\delta f,\delta q\in L^2(\Omega)$. The main result concerning perturbed data is given in the following theorem.
\begin{theorem}\thlabel{prime}
    Let $u\in H^s(\Omega)$ be the solution to the continuous problem and $(u_h,z_h)$ be the solutions to system (\ref{smoothsys}). Let $0\le \alpha\le 1, \eta\tau = 0$ and both non-negative. Let $B$ be a region satisfying estimate in \thref{co2} and \thref{ultraweak}; $\kappa,\epsilon$ be the parameters therein. Then
    $$||u-u_h||_{H^1(B)}\le C(P)h^{\kappa (s-1)}\left(||u||_{H^s(\Omega)}+Pt\right).$$
    Moreover, if $\alpha=1$, $\tau = 2$ and $\eta=0$, then we have 
    \begin{equation*}
        ||u-u_h||_B\le C(P)h^{\kappa s}\left(||u||_{H^s(\Omega)}+Pt\right),
    \end{equation*}
    where $Pt = h^{-(s-1)-\alpha}||\delta q||_\omega+h^{-(s-1)-\frac{\tau}{2}}||\delta f||_\Omega$.
\end{theorem}
\thref{prime} shows that the level of perturbation depends on the choice of parameters. Specifically, if we want to achieve an optimal convergence in the $L^2$-norm, the necessary parameter choice makes the system more sensitive to perturbations. A more complete discussion follows at the end of this section, after the proof of \thref{prime}. 

First, we give a stability estimate of the linear system in terms of its condition number.
\begin{prop}\thlabel{con}
    Let $0\le\alpha\le 1$, $\eta\ge 0$, and $0\le\tau\le 2(s-1)$. Then, the linear system (\ref{smoothsys}) has a Euclidean condition number 
    $$\mathcal{K}_2=C(1+h^2||P||_{L^\infty})h^{-2s}.$$
\end{prop}
\begin{proof}
    Let $A$ be as defined in (\ref{A}). First, we have for $\phi_h,\psi_h\in V^p_h$ or $W_h^p$,
    \begin{equation*}
        a(\phi_h,\psi_h)\le \frac{C(1+h^2||P||^2_{L^\infty})}{h^2}||\phi_h||_\Omega||\psi_h||_\Omega,
    \end{equation*}
    by (\ref{inv}). Then from \thref{boundA} we know that
    \begin{equation*}
        \sup_{\substack{(u_h,z_h)\in V_h^p\times W_h^p\\(v_h,w_h)\in V_h^p\times W_h^p}}\frac{A_h[(u_h,z_h),(v_h,w_h)]}{||(u_h,z_h)||_\Omega||(v_h,w_h)||_\Omega}\le C\frac{1+h^2||P||^2_{L^\infty}}{h^2}.
    \end{equation*}
    Moreover, from \thref{BNBA}, we know that for some specific $(v_h,w_h)$, we have
    \begin{equation*}
        A_h[(u_h,z_h),(v_h,w_h)]= ||u_h||^2_\omega+s_h(u_h,u_h)+s_h^*(z_h,z_h).
    \end{equation*}
    Therefore, for all $(u_h,z_h)\in V_h^p\times W_h^p$, we have
    \begin{equation*}
        \sup_{(v_h,w_h)\in V_h^p\times W_h^p}\frac{A_h[(u_h,z_h),(v_h,w_h)]}{||(v_h,w_h)||_\Omega} \ge Ch^{2(s-1)}||(u_h,z_h)||_\Omega.
    \end{equation*}
    Now by \cite[Theorem 3.1]{ern2006evaluation} we have the desired estimate, where $C(P)=C(1+h^2||P||^2_{L^\infty})$.
    
\end{proof}

Our aim now is to show that $u_h$ converges to $u$ in the $L^2$-norm on $\omega$ and the $H^{-1}$-norm on $\Omega$. The steps follow the same structure as in the last section.
\begin{lemma}\thlabel{consis2}
    Let $\pi_{sz}u$ be the Scott-Zhang interpolant of $u$ on $V^p_h$. Then  we have for some constant $\color{red}C>0$, such that $\forall v_h\in V^p_h$,
    \begin{equation*}
    \begin{aligned}
         s_h(\pi_{sz}u,v_h)-&h^2(f_h,\mathcal{L}_hv_h)_\Omega \\
         \le &C\left((1+h^2||P||_{L^\infty})h^{s-1}||u||_{H^s(\Omega)}+h^{s-1}||f||_{H^{s-2}(\Omega)}\right)|||(v_h,0)|||.
         \end{aligned}
    \end{equation*}
\end{lemma}    
\begin{proof}
    First we have
\begin{equation*}
    \begin{aligned}
    s(\pi_{sz}u,v_h)-h^2(f_h,\mathcal{L}_hv_h)=&\mathcal{J}_h(\pi_{sz}u,v_h)+h^2(\mathcal{L}_h\pi_{sz}u-f_h,\mathcal{L}_hv_h)_\Omega\\
    &+h^{2(s-1)}\langle\pi_{sz}u,v_h\rangle_\Omega.
    \end{aligned}
\end{equation*}
The first quantity is bounded by using \thref{weakcons} directly and the third quantity is easily bounded by using the Cauchy-Schwarz inequality. For the second quantity, we have
\begin{equation*}
    \begin{aligned}
        h^2(\mathcal{L}_h\pi_{sz}u-f_h,\mathcal{L}_hv_h)_\Omega & = h^2(\mathcal{L}_h(\pi_{sz}u-u),\mathcal{L}_hv_h)_\Omega+h^2(f-f_h,\mathcal{L}_hv_h)_\Omega\\
        &\le Ch^{s-1}\left(C(P)||u||_{H^s(\Omega)}+||f||_{H^{s-2}(\Omega)}\right)||h\mathcal{L}_hv_h||_\Omega,
    \end{aligned}
\end{equation*}
where $C(P) = 1+h^2||P||_{L^\infty}$. The estimate follows from the fact that $\pi_{sz}$ has an optimal convergence property in all Sobolev norms and $\pi_h$ has an optimal convergence property for the $L^2$-norm ($f_h=\pi_hf$ when $f\in L^2(\Omega)$).
Thus we obtain
\begin{equation*}
    h^2(\mathcal{L}_h\pi_{sz}u-f_h,\mathcal{L}_hv_h)_\Omega \le C(1+h^2||P||_{L^\infty})h^{s-1}||u||_{H^s(\Omega)}||h\mathcal{L}_hv_h||_\Omega.
\end{equation*}
Then by combining the two cases above we get the desired estimate
\begin{equation}\label{sherr}
    s_h(\pi_{sz}u,v_h)-h^2(f_h,\mathcal{L}_hv_h)_\Omega\le C(1+h^2||P||_{L^\infty})h^{s-1}||u||_{H^s(\Omega)}|||(v_h,0)|||.
\end{equation}
\end{proof}
\begin{remark}
    Notice that, comparing with \thref{consis1}, \thref{consis2} requires the source term $f$ to have a higher order regularity $f\in H^{s-2}(\Omega)$. 
\end{remark}
\begin{lemma}\thlabel{stabnormest2}
Let $(u_h,z_h)$ be the solution to (\ref{smoothsys}) and let $u\in H^s(\Omega)$ be the solution to the continuous problem. Let $\alpha\le 1$, and $\tau\eta=0$. Then
\begin{equation*}
    |||(u_h-\pi_{sz}u,z_h)|||\le C\left((1+h||P||_{L^\infty})h^{s-1}||u||_{H^s(\Omega)}+h^{s-1}||f||_{H^{s-2}(\Omega)}+Pt\right),
\end{equation*}
where $Pt = h^{-\alpha}||\delta q||_\omega+h^{-\frac{\tau}{2}}||\delta f||_\Omega$.
\end{lemma}
\begin{proof}
The proof follows using the same arguments as \thref{stabnormest1}. 

\end{proof}
\begin{co}\thlabel{interiorest}
    Let $\alpha\ge 0$. The discrete solution $u_h$ converges optimally in $\omega$ if the measurements are free of perturbation. In addition the $H^1$ error in the whole domain $\Omega$ is bounded. More precisely,
    \begin{equation*}
       ||u-u_h||_\omega \le C\left((1+h||P||_{L^\infty})h^{s-1}||u||_{H^s(\Omega)}+||\delta q||_\omega+h^{-(\frac{\tau}{2}-\alpha)}||\delta f||_\Omega\right),
    \end{equation*}
    and
    \begin{equation*}
        ||u-u_h||_{H^1(\Omega)} \le C((1+h||P||_{L^\infty})||u||_{H^s(\Omega)}+Pt),
    \end{equation*}
    where $P_t = h^{-\alpha}||\delta q||_\omega+h^{-\frac{\tau}{2}}||\delta f||_\Omega$.
\end{co}
\begin{proof}
    From the definition of $|||\cdot|||$ as (\ref{stabnorm}), we have
    \begin{equation*}
    \begin{aligned}
        ||u-u_h||_\omega &\le ||u-\pi_{sz}u||_\omega +||\pi_{sz}u-u_h||_\omega\\
        &\le Ch^{s-1}||u||_{H^s(\Omega)}+h^{\alpha}|||(\pi_{sz}u-u_h,0)|||\\
        &\le C\left((1+h||P||_{L^\infty})h^{s-1}||u||_{H^s(\Omega)}+||\delta q||_\omega+h^{-(\frac{\tau}{2}-\alpha)}||\delta f||_\Omega\right),
        \end{aligned}
    \end{equation*}
    and
    \begin{equation*}
        \begin{aligned}
        ||u-u_h||_{H^1(\Omega)} &\le ||u-\pi_{sz}u||_{H^1(\Omega)} + ||\pi_{sz}u-u_h||_{H^1(\Omega)} \\
        &\le C||u||_{H^1(\Omega)}+h^{-(s-1)}|||(\pi_{sz}u-u_h,0)|||\\
        &\le C((1+h||P||_{L^\infty})||u||_{H^s(\Omega)}+Pt),
        \end{aligned}
    \end{equation*}
   where the last inequality comes from \thref{stabnormest2}.
    
\end{proof}

Now, with the conditional stability estimates given in Section \ref{sec:continuous}, we can prove \thref{prime} by using the same arguments in the proof of \thref{ultraweakconv} and \thref{prime12}.\\

We discuss here the choice and trade-off between convergence rate and stability. In terms of the $H^1$ error, when $s>1$, the choice of $\eta$ is robust. We could choose a set of parameters that depend only on the regularity $s$ working for both lower order and higher order solutions, namely, by letting $0<\alpha\le 1, \eta =0$ and $0<\tau\le2$. However, from \thref{prime}, it is clear that the optimal choice for $\alpha$ and $\tau$ is $\alpha = \tau =0$, to minimize the error caused by perturbations. Now, if we control $\tau = 0$, then any large number can be taken for $\eta$ to generate convergence according to \thref{prime}. Specifically, if $\eta\to \infty$, we get a reduced system, namely, $s^*_\eta(z_h,w_h) = 0$ and $s^*_h(z_h,w_h) = h^\tau\langle z_h,w_h\rangle_\Omega$.\\

In terms of the $L^2$ error, we emphasize that there is always a trade-off between convergence rate and stability. Indeed a comparison of \thref{H-2norm} and \thref{prime}. \thref{H-2norm} show the trade-off between the convergence rate and the condition number of the linear system. The result of \thref{prime} shows the tension between the convergence rate and the error growth rate caused by inaccuracies in the data measurement. Based on the above discussion, we give a table for different parameters chosen for considering optimal convergence in the $L^2$- and $H^1$-norm respectively.
\begin{table}[h!]
\centering
\caption{Parameters for optimal convergence in $L^2$ and $H^1$ norm.}
\begin{tabular}{ | c | c| c | c| c|} 

  \hline
     & $\alpha$ & $\eta$ & $\tau$ & Stability \\ 
  \hline
  $L^2$ error & $\alpha = 1$ & $\eta = 0$ & $ \tau = 2$ & Worse stability under numerical integration/perturbation \\ 
  \hline
  $H^1$ error& $\alpha = 0$ & $\eta \ge 0$ & $\tau=0$ & N/A \\ 
  \hline
\end{tabular}

\label{table:1}
\end{table}

Now, we give an estimate of the effect of interpolation error, numerical integration or rounding error in computing the stiffness matrix of the linear system (\ref{smoothsys}). Denote the theoretical solution to (\ref{smoothsys}) still by $u_h$ and the computational solution by $\tilde u_h$. Let ($A_h$)$A_{h\bar{e}}$ be the matrix (non)perturbed by rounding errors and numerical integration, and assume 
\begin{equation}\label{cdtnum}
||A_h-A_{h\bar{e}}||\le \bar{e}||A_h||, 
\end{equation} 
By numerical linear algebra, with shape functions chosen to have a mass matrix $\mathcal{M}_h$ with condition number $\mathcal{K} (\mathcal{M}_h)$, that is uniform with $h$ due to quasi-uniformity of $\mathcal{T}_h$. Let $U_h$ and $\tilde U_h$ be the coefficients of $u_h$ and $\tilde u_h$ in $V^p_h$, we have
\begin{equation}\label{condi}
\begin{aligned}
||u_h-\tilde u_h||_\Omega & \le \sigma^\frac{1}{2}_{max}(\mathcal{M}_h)||U_h-\tilde U_h||_N \\
    &\le \sigma^\frac{1}{2}_{max}(\mathcal{M}_h)\mathcal{K}(A_h)||U_h||_N \\
    &\le \sigma^\frac{1}{2}_{max}(\mathcal{M}_h)\mathcal{K}(A_h) \sigma_{min}^{-\frac{1}{2}}(\mathcal{M}_h)||M^\frac{1}{2}U_h||_N \\
    &\le C(1+h^2||P||^2_{L^\infty}) \mathcal{K}^\frac{1}{2}(\mathcal{M}_h)\frac{\bar{e}}{h^{2s}}||u_h||_\Omega.
\end{aligned}
\end{equation}
\begin{co}[Error estimate with interpolation error]\thlabel{err}
    Let $\alpha = \tau = 0$ and $\eta\to \infty$. Let $\bar{e} $ be the error with respect to solving linear equations (rounding errors, interpolation errors, etc), and $\tilde u_h$ be the solution considering $\bar{e}$ , then by (\ref{condi}):
    $$||u-\tilde u_h||_{H^1(B)}\le C(P)h^{\kappa(s-1)}\left(||u||_{H^{s}(\Omega)}+h^{-(s-1)}||\delta g||_\omega+h^{-(s-1)}||\delta f||_\Omega\right)+ R(h),$$
    where 
    $$R(h)=C(P)\bar{e}\left(\frac{1}{h^{2s+1}}||u||_{H^{s}(\Omega)}+\frac{1}{h^{3s}}||\delta f||_\Omega+\frac{1}{h^{3s}}||\delta q||_\omega\right),$$
    and $C(P) =\exp\left(C\kappa^{1-\epsilon}(||P||^{\frac{2}{3}}_{L^\infty(\Omega)}+1)\right)$
\end{co}
\begin{proof}
    The additional error for $||u-\tilde u_h||_{H^1(B)}$ is $||u_h-\tilde u_h||_{H^1(B)}$. Applying the inverse inequality (\ref{inv}) and the condition number bound of \ref{condi} we have
    \begin{equation*}
        \begin{aligned}
            ||u_h-\tilde u_h||_{H^1(B)} &\le \frac{C}{h}||u_h-\tilde u_h||_\Omega\\
            &\le \frac{C(1+h^2||P||^2_{L^\infty})\bar{e}}{h^{2s+1}}||u_h||_\Omega\\
            &\le \frac{C(1+h^2||P||^2_{L^\infty})\bar{e}}{h^{2s+1}}(||u_h-u||_\Omega+||u||_\Omega)\\
            &\le C'(P)\bar{e} (\frac{1}{h^{2s+1}}||u||_{H^{s}(\Omega)}+\frac{1}{h^{3s}}||\delta f||_\Omega+\frac{1}{h^{3s}}||\delta q||_\omega),
        \end{aligned}
    \end{equation*}
   where $C'(P) = \exp\left(C\kappa^{1-\epsilon}(||P||^{\frac{2}{3}}_{L^\infty(\Omega)}+1)\right)$. The last inequality uses the bound for the $||u_h-u||_\Omega$ in \thref{interiorest}. The corollary then follows.
    
\end{proof}

\section{Numerical experiments}
\label{sec:experiments}
We present numerical experiments for the Schr\"odinger unique continuation problem (\ref{pbl}). First, we consider $\Omega$ to be the unit disk, and the solution function to be a piecewise linear function which is not smooth at $y =0$.
\begin{ex}\thlabel{H-1}
    Let $f\in H^{-1}(\Omega)$ to be defined as
    \begin{equation*}
        <f,w>_\Omega = \int_{-1}^1w(x,0)dx,
    \end{equation*}
    and consider the piece-wise linear function 
    \begin{equation*}
        u_0(x,y) = \left\{\begin{aligned}
            -y && \texttt{$y>0$}\\
            0 && \texttt{else},
        \end{aligned}
        \right.
    \end{equation*}
    It is easy to observe that $u_0$ is the solution to
    \begin{equation*}
        \left\{\begin{aligned}
            -\Delta u = f && \texttt{in $\Omega$}\\
            u=u_0 && \texttt{in $\omega$}.
        \end{aligned}
        \right.
    \end{equation*}
    In this example we 
    define the geometries of interest to be
    \begin{equation*}
    \omega = \{(x,y)\in \Omega|x>0,x^2+y^2>0.25\}
    \end{equation*}
    and
    \begin{equation*}
    B = \{(x,y) \in \Omega| x>0, x^2+y^2 > 0.0625\}
\end{equation*}
\end{ex}
\begin{figure}[h]
    \centering
    \caption{Domain in \thref{H-1}}
    \includegraphics[width = 5cm]{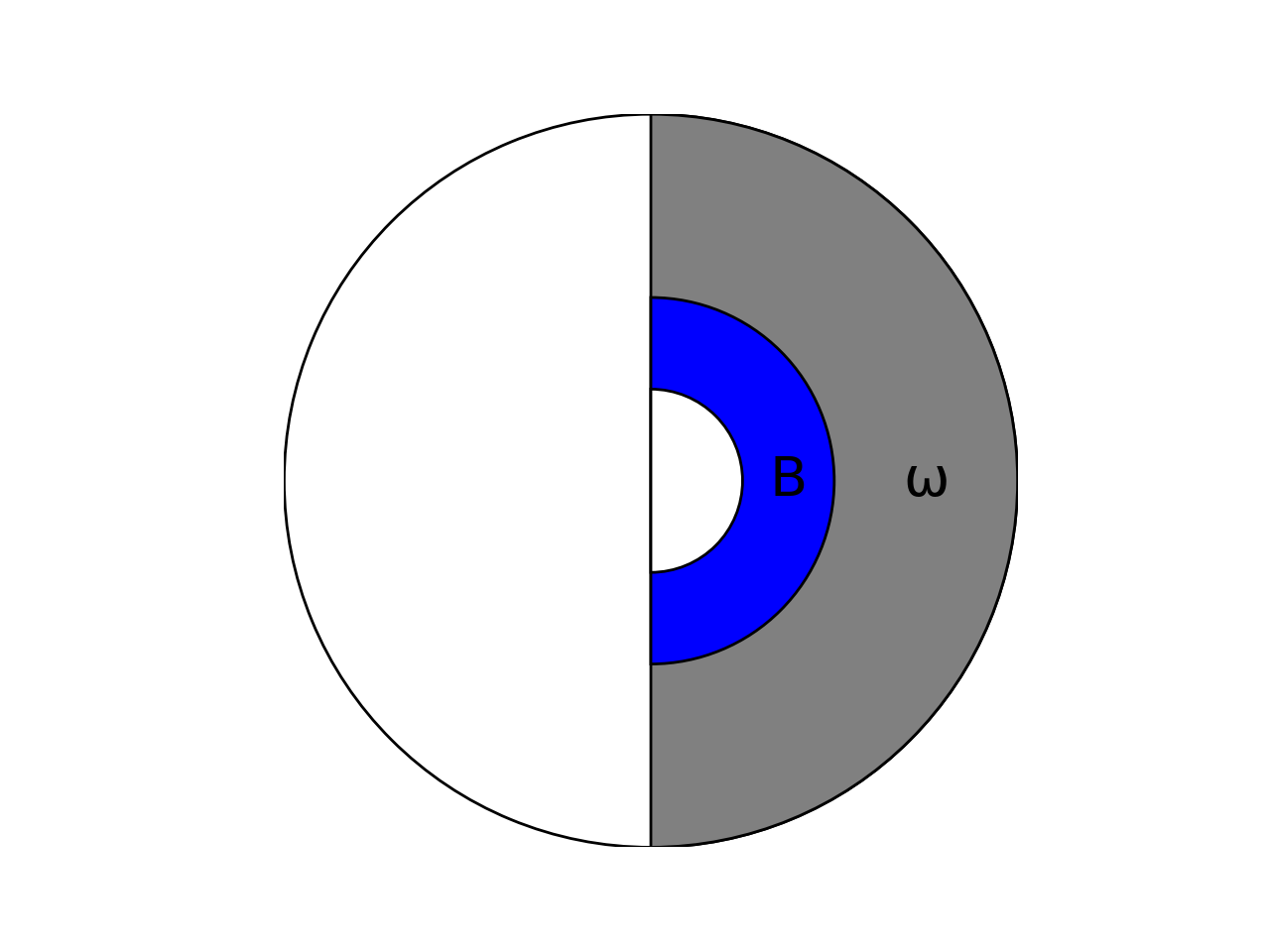}
    \label{fig:diskdomain}
\end{figure}
\begin{figure}[h]
    \centering
    \caption{Error comparison of schemes with $\eta = 0$ and $\tau = 0$}%
    \subfloat[\centering  $\eta =0, \alpha =1$ and $\tau \to \infty$.]{{\includegraphics[width=5.8cm]{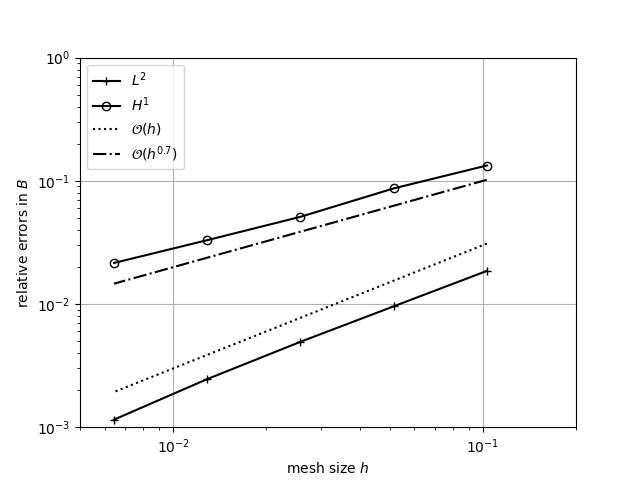} }}%
    \qquad
    \subfloat[\centering  $\eta \to \infty, \alpha = 1$ and $\tau =0$.]{{\includegraphics[width=5.8cm]{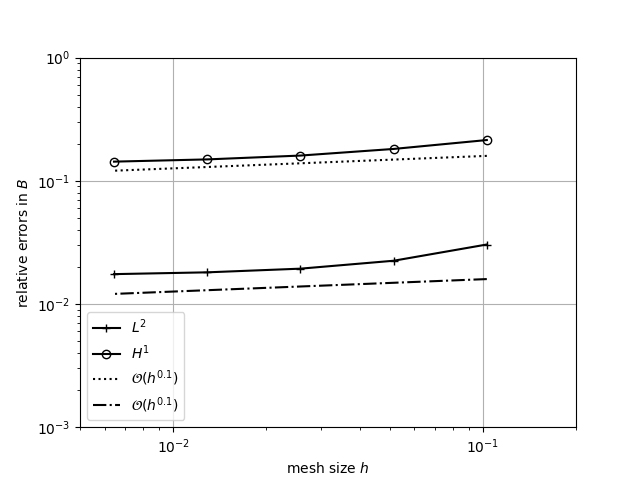} }}%
    \label{fig3.00}
\end{figure}
According to \thref{prime12}, $\eta\tau =0$ is needed to generate optimal convergence with respect to regularity $s$, but by \thref{cvg}, $\eta=0$ is needed to generate weak convergence for $H^1$-only case. We show by comparison that this does strengthen the convergence. In Figure \ref{fig3.00} it is observed that the simplified scheme obtained by setting $s_h^*(z_h,w_h) = \langle z_h,w_h\rangle_\Omega$ is generating much poorer convergence in both the $L^2$ and the $H^1$-norm than the scheme with dual stabilization terms defined by (\ref{dualstabeta}) with $\eta = 0$. This indicates that, from Table \ref{table:1}, the robustness for $\eta$ does not hold for rough $u$.\\
Now we show the influence of the parameter $\tau$ on the convergence rate. In this example, we control $\alpha = 1, \eta=0$. 
\begin{figure}[h]
    \centering
    \caption{Different parameters achieving increasing convergence}%
    \subfloat[\centering $\tau = 0$.]{{\includegraphics[width=5.8cm]{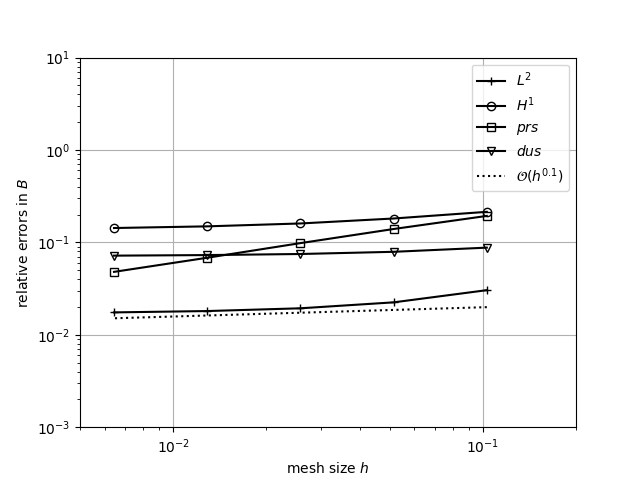} }}%
    \qquad
    \subfloat[\centering $\tau = 0.5$.]{{\includegraphics[width=5.8cm]{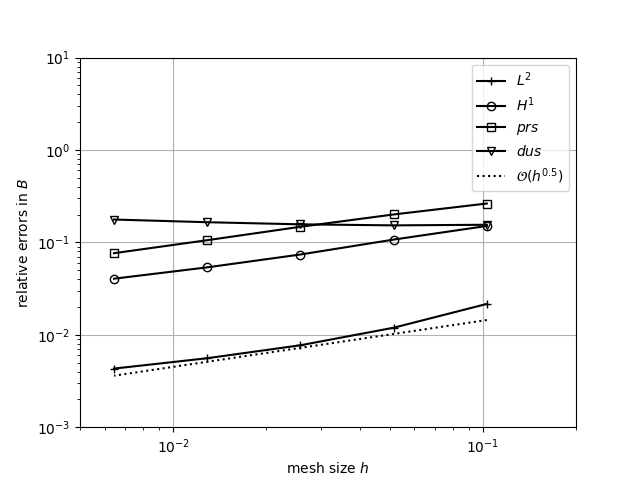} }}%
    \qquad
    \subfloat[\centering $\tau = 2$.]{{\includegraphics[width=5.8cm]{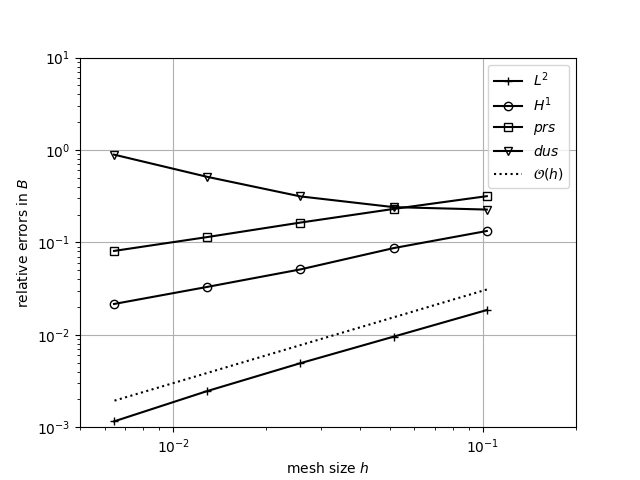} }}%
    \qquad
    \subfloat[\centering $\tau = 10$.]{{\includegraphics[width=5.8cm]{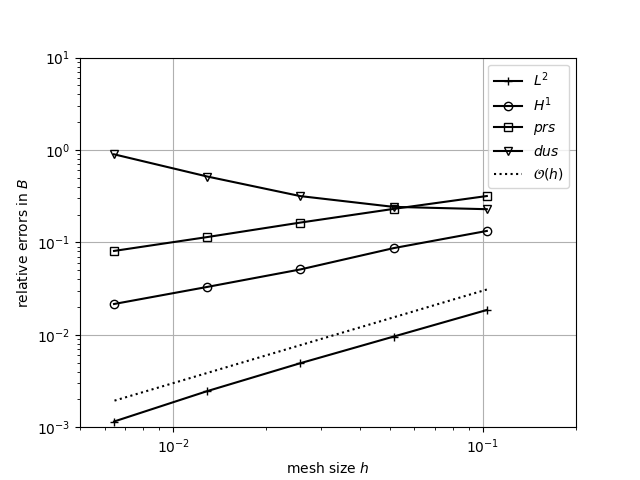} }}%
    \label{fig3.5}
\end{figure}
In Figure \ref{fig3.5}, we compare the convergence rate with respect to the $L^2$ and the $H^1$-norm for varying values of the parameter $/tau$. Optimal convergences are achieved for $\frac{\tau}{2}=1$. This is consistent with \thref{cvg}. $prs$ and $dus$ represent $\sqrt{\mathcal{J}_h(u_h,u_h)}$ and $||z_h||_\Omega$ respectively. It is shown in part (A) that $||z_h||_\Omega$ is bounded for $\tau = 0$ and may not be bounded when $\tau>0$, which is consistent with \thref{bdness1}. \\
\newpage
In the second example, we present numerical experiments for real Schrödinger unique continuation problem (\ref{pbl}) with the high order $H^1$-conforming method (\ref{sys}). We will focus on the classical Hadamard example for ill-posed elliptic equations with logarithm potential. 
\begin{ex}
Consider the unique continuation problem
\begin{equation}\label{exa}
    \left\{\begin{array}{rll}
        -\Delta u+10\log(y+\frac{1}{2})u &= 10\log(y+\frac{1}{2})\sin{x}\sinh{y}&\text{in $\Omega:(0,\pi)\times(0,1)$,} \\
        u(x,y) &= \sin(x)\sinh(y) &\text{for $(x, y)\in \omega$}
    \end{array}\right.
\end{equation}
whose solution is given by
\begin{equation*}
    u(x,y) = \sin(x)\sinh(y).
\end{equation*}
\end{ex}
For such Hadamard-type solutions, we consider the interior datum $q=u|_\omega$ and study the convergence in the target set $B$ for two geometric settings of $\omega$ and $B$, namely
\begin{equation}\label{conv}
    \omega =\Omega\backslash [\frac{\pi}{4},\frac{3\pi}{4}]\times [0.05,1],\;\; B=\Omega\backslash [\frac{\pi}{4},\frac{3\pi}{4}]\times [0.75,1],
\end{equation}
and 
\begin{equation}\label{nonconv}
    \omega =(\frac{\pi}{4},\frac{3\pi}{4})\times (0.05,0.5),\;\; B=(\frac{\pi}{8},\frac{7\pi}{8})\times (0.05,0.75).
\end{equation}

\begin{figure}[h]
    \centering
    \caption{Different domains.}%
    \subfloat[\centering Domain (\ref{conv}).]{{\includegraphics[width=5.8cm]{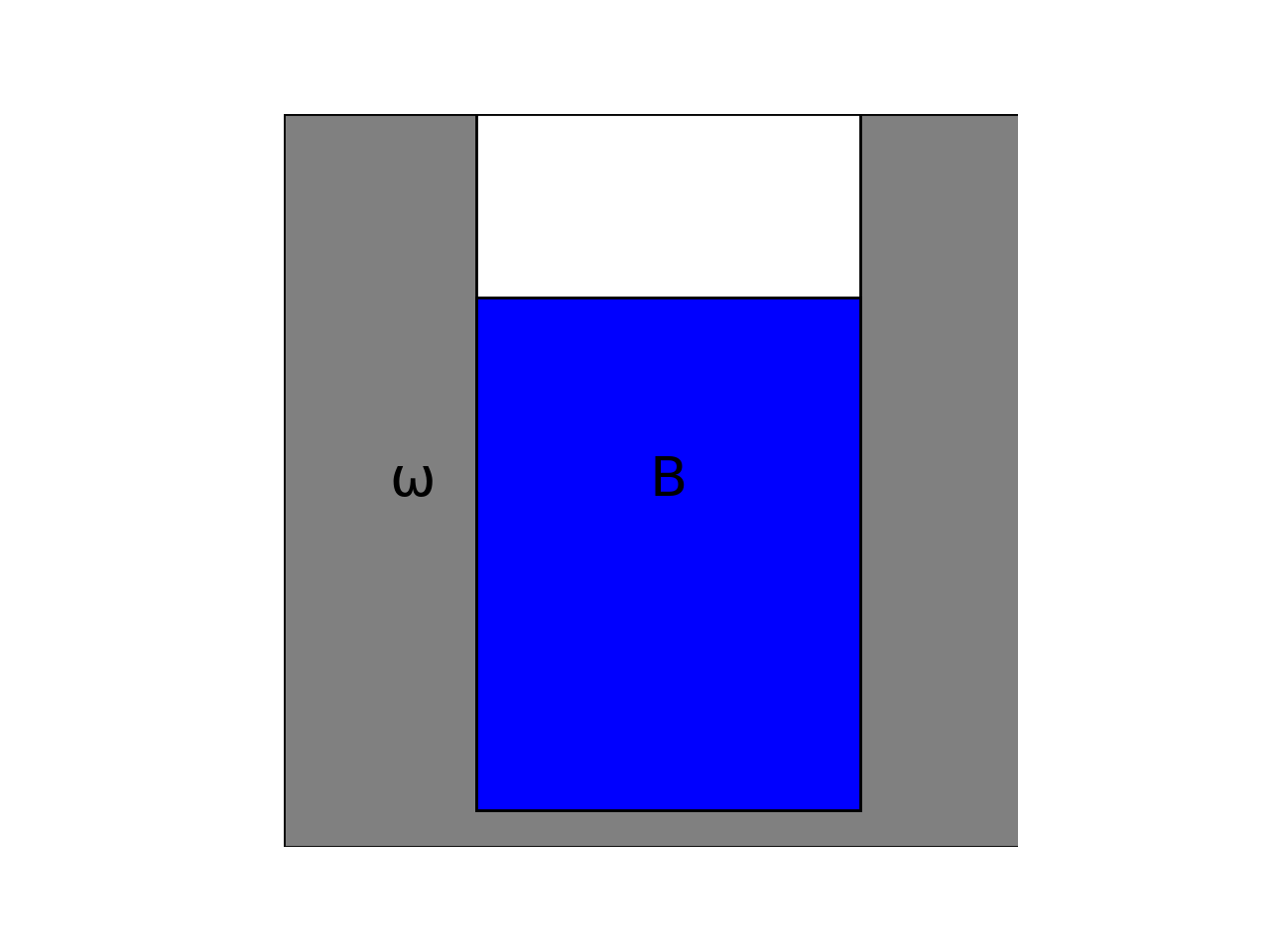} }}%
    \qquad
    \subfloat[\centering Domain (\ref{nonconv}).]{{\includegraphics[width=5.8cm]{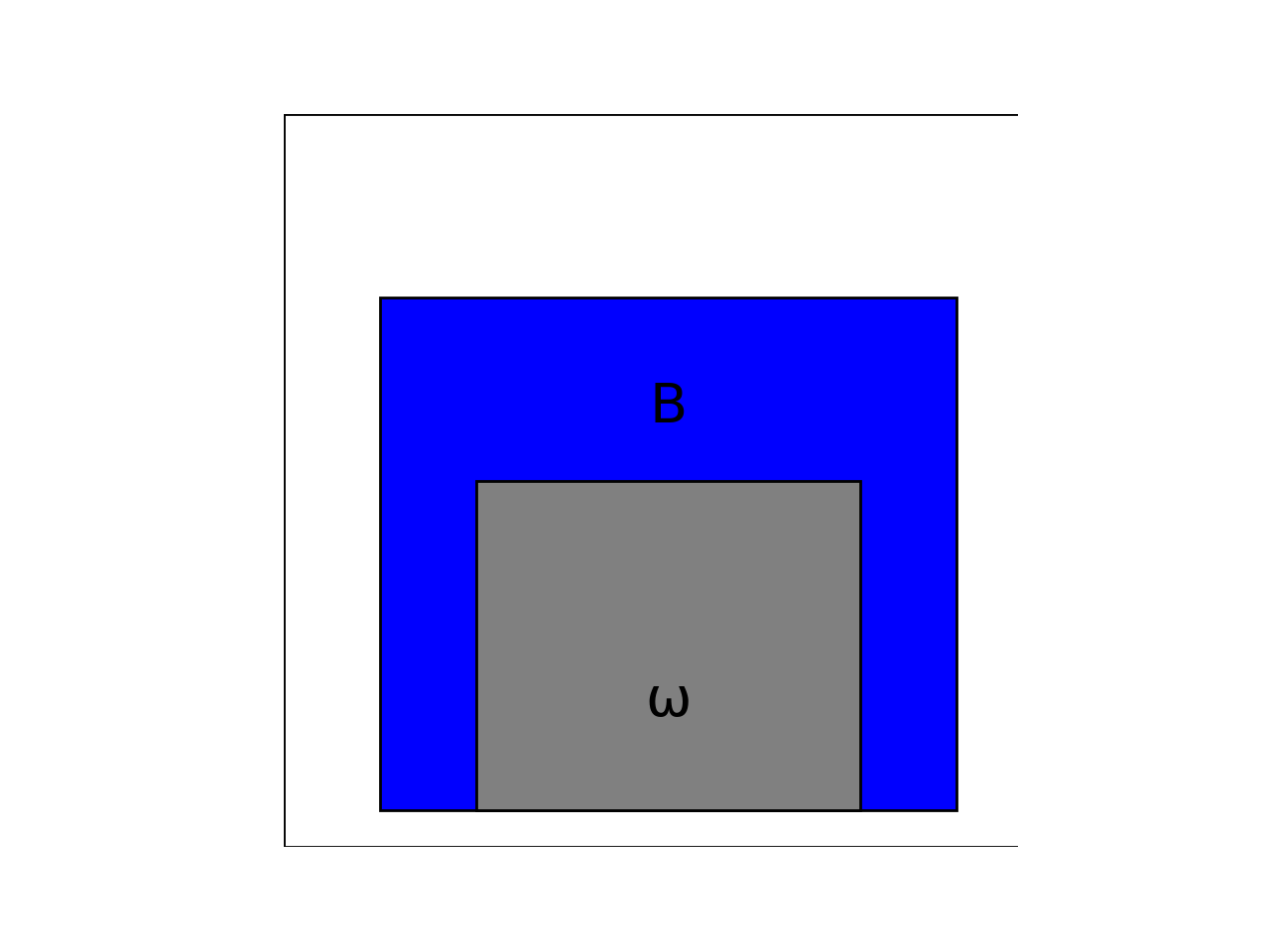} }}%
    \label{fig:domains}%
\end{figure}

In domain (\ref{conv}) we examined the $H^1$-error in the region of interest $B$ using three different orders of finite element space. It can be observed from Figure \ref{fig3.6} that the results of the schemes with or without additional dual stabilizer $s^*_\eta(\cdot, \cdot)$ (as defined in (\ref{dualstabeta})) almost coincides. This illustrates the robustness of $\eta$ for higher regularity solutions.\\
\begin{figure}[h]
    \centering
    \caption{Convergence comparison by letting $\eta = 0$ and $\eta\to \infty$}%
    \subfloat[\centering $\eta \to \infty$.]{{\includegraphics[width=5.8cm]{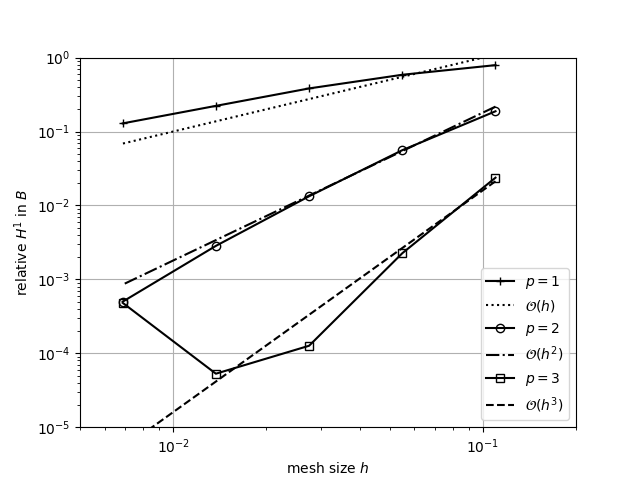} }}%
    \qquad
    \subfloat[\centering $\eta =0$.]{{\includegraphics[width=5.8cm]{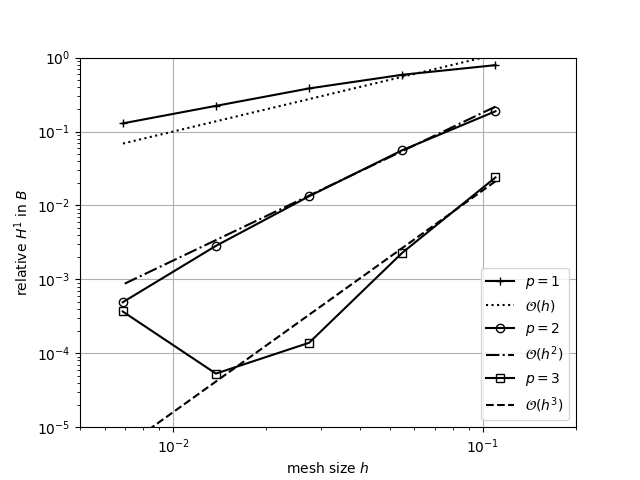} }}%
    \label{fig3.6}
\end{figure}

Now we give a convergence comparison for different domains. In this example we set $\alpha  = \tau =0$ and $\eta\to \infty$.\\
\begin{figure}[t]
    \centering
    \caption{Convergence comparison for different domains.}%
    \subfloat[\centering Convergence in domain (\ref{conv}).]{{\includegraphics[width=5.8cm]{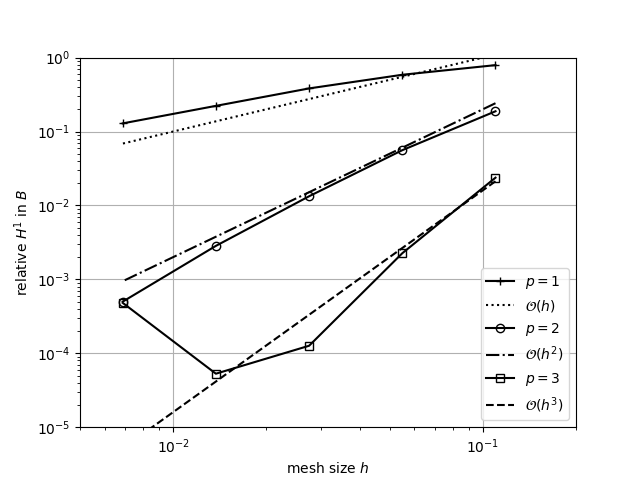} }}%
    \qquad
    \subfloat[\centering Convergence in domain (\ref{nonconv}).]{{\includegraphics[width=5.8cm]{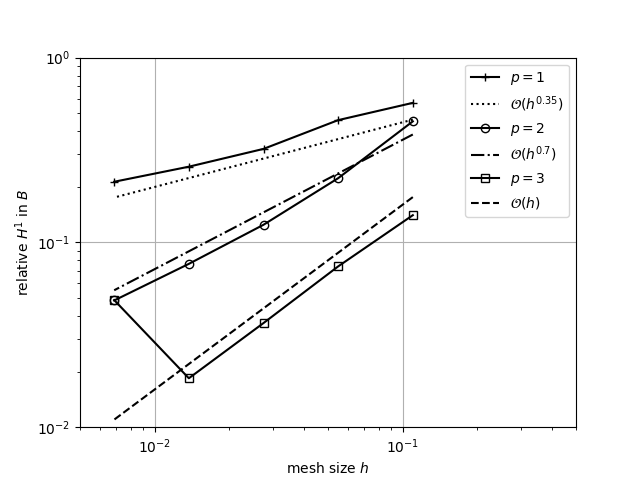} }}%
    \label{fig2}%
\end{figure}

In Figure \ref{fig2}, we observe that for domain (\ref{conv}), the convergence rate for each choice of order $p$ almost achieves optimal convergences, which means that in \thref{prime}, $\kappa\approx 1$, while for domain (\ref{nonconv}), $\kappa$ is around $0.35$. It is also worth noting that, for both cases, the finite element approximation using $p=3$ space has a turning point for convergence although no perturbation of data is imposed. This is probably due to the ill-posedness (high condition number) of the linear system system (\ref{smoothsys}) and the associated amplification of quadrature error in the numerical integration, which is consistent with \thref{err}.
\section*{Acknowledgments}
E.B. was supported by EPSRC grants EP/T033126/1 and EP/V050400/1. L.O. was supported by the European Research Council of the European Union, grant 101086697 (LoCal), and the Reseach Council of Finland, grants 347715 and 353096. Views and opinions expressed are those of the authors only and do not necessarily reflect those of the European Union or the other funding organizations.

\appendix

\section{Proof of \thref{tbi}}
\label{sec:prooftbi}
To prove \thref{tbi}, we need the following Lemma
\begin{lemma}\thlabel{log}
    Suppose that $A$, $B$, $D$, $c$ be positive numbers and $\kappa\in (0,1)$ satisfying $D\le B$ and $D\le e^{c(1-\kappa)\lambda}A+e^{-c\kappa\lambda}B$ for all $\lambda\ge\lambda_0>0$. Then for any $\epsilon$ positive and small there exists constant $C$ depending on $c,\epsilon$ such that
    \begin{equation*}
        D\le  e^{C\kappa^{1-\epsilon}(\lambda_0+1)}A^\kappa B^{1-\kappa}.
    \end{equation*}
\end{lemma}
\begin{proof}
See \cite[Lemma 5.2]{le2012carleman}.
\end{proof}

\begin{proof}[Proof of \thref{tbi}]
The proof is similar to \cite[Corollary 1]{burman2020stabilized}. Due to the density of $C^2(\Omega)$ in $H^2(\Omega)$, it suffices to show the inequality holds for $u\in C^2(\Omega)$. Let $r_0<r_1<r_2<r_3<r_4$ and $B_j\subset \Omega$, $j=0,1,2,3,4$. Choose $\rho(x)=-d(x,x_0)$ and $K=\Bar{B_4}\backslash B_0$. Notice that $\rho$ has no critical point in $K$. Let $\chi\in C_0^\infty(K)$ satisfies $\chi =1$ in $B_3\backslash B_1$, and $w=\chi u$. Then according to Theorem \thref{car} we have
    \begin{equation*}
        \begin{aligned}
             \int_{B_4\backslash B_0}(\tau^3|w|^2+\tau|\nabla w|^2)e^{2\tau\phi}dx &\le C\int_{B_4\backslash B_0}|\Delta w|^2e^{2\tau\phi}dx\\
             &\le 2C\int_{B_4\backslash B_0}|-\Delta w+Pw|^2e^{2\tau\phi}dx\\
             &+2C\int_{B_4\backslash B_0}|Pw|^2e^{2\tau\phi}dx.
        \end{aligned}
    \end{equation*}
    Set $\tau_0=(4C||P||^2_{L^\infty(\Omega)})^{\frac{1}{3}}+1$, we obtain
    \begin{equation}\label{co1}
        \int_{B_4\backslash B_0}(\frac{\tau^3}{2}|w|^2+\tau|\nabla w|^2)e^{2\tau\phi}dx \le 2C\int_{B_4\backslash B_0}|-\Delta w+Pw|^2e^{2\tau\phi}dx.
    \end{equation}
    Denote $\Phi(r)=e^{-\alpha r}$, and then we can bound the left-hand side in (\ref{co1}) by
    \begin{equation*}
        \int_{B_2\backslash B_1}(\tau^3|w|^2+\tau|\nabla w|^2)e^{2\tau\phi}dx \ge \tau e^{2\tau\Phi(r_2)}||u||^2_{H^1(B_2)}-\tau e^{2\tau}||u||^2_{H^1(B_1)}
    \end{equation*}
    since $\phi\le 1$ and $\tau_0>1$.
    Moreover, the right-hand side can be bounded by
    \begin{equation*}
    \begin{aligned}
         \int_{B_4\backslash B_0}|-\Delta u+Pu|^2e^{2\tau\phi}dx+\int_{(B_4\backslash B_3)\cup B_1}|[\Delta,\chi]u|^2e^{2\tau\phi}dx\\
         \le e^{2\tau}||-\Delta u+Pu||^2_{L^2(\Omega)}+e^{2\tau\Phi(r_3)}||u||^2_{H^1(\Omega)}+e^{2\tau}||u||^2_{H^1(B_1)}.
    \end{aligned}
    \end{equation*}
    Combining the two bounds we obtain
    \begin{equation}\label{es}
        \tau e^{2\tau\Phi(r_2)}||u||^2_{H^1(B_2)}\le C\tau e^{2\tau}(||u||^2_{H^1(B_1)}+||-\Delta u+Pu||^2_{L^2(\Omega)})+Ce^{2\tau\Phi(r_3)}||u||^2_{H^1(\Omega)}.
    \end{equation}
    Then by setting $\tau_0 = C(1+||P||^{\frac{2}{3}}_{L^\infty(\Omega)})$ for some large $C$ to absorb the constant before $||u||_{H^1(B_2)}$, we have
    \begin{equation*}
        ||u||_{H^1(B_2)}\le Ce^{\tau(1-\Phi(r_2))}(||u||_{H^1(B_1)}+||-\Delta u+ Pu||_{L^2(\Omega)})+Ce^{\tau(\Phi(r_3)-\Phi(r_2))}||u||_{H^1(\Omega)}.
    \end{equation*}
    The conclusion follows from \thref{log} by letting 
    \begin{equation*}
        c=1-\Phi(r_3), \;\; \kappa=\frac{\Phi(r_2)-\Phi(r_3)}{1-\Phi(r_3)}\in(0,1).
    \end{equation*}
\end{proof}

\section{Proof of \thref{co2} and \thref{ultraweak}}
\label{sec:proofco2}
To prove \thref{co2}, we need the following Lemma
\begin{lemma}\thlabel{aux}
    Let $\Omega$ be any polygonal/polyhedral domain. Then the auxiliary Schrödinger equation with complex Robin boundary is well-posed, i.e, the system
    \begin{equation}\label{auxiliary}
        \left\{\begin{array}{rcc}
        -\Delta u+Pu & =f & \texttt{in $\Omega$,}   \\
        \partial_n u+iu & =0 & \texttt{on $\partial \Omega$,}
    \end{array}\right.
    \end{equation}
    has a unique solution for $f\in (H^1(\Omega))'$, and
    $$||u||_{H^1(\Omega)}\le C(P)||f||_{(H^1(\Omega))'}.$$
    Moreover, if $P$ is positive in $\Omega$, then the constants $C(P)$ are independent of $P$.
\end{lemma}
\begin{proof}[proof of \thref{aux}]
    The idea follows \cite[Section 8.1]{melenk1995generalized}. To show existence and uniqueness of the problem, it suffices to show that if the following primal and dual weak formulation has a solution then it is unique. That is, there is a unique $u\in H^1(\Omega)$, such that for all $v\in H^1(\Omega)$, 
    \begin{equation*}
        \int_\Omega \nabla u\cdot \nabla \bar{v}dx+\int_\Omega Pu\bar{v}dx \pm i\int_{\partial\Omega}u\bar{v}dS=\langle f, \bar{v}\rangle.
    \end{equation*}
    To show this, consider the difference $w$ between two solutions, which satisfies
    \begin{equation*}
        \int_\Omega \nabla w\cdot \nabla \bar{v}dx+\int_\Omega Pw\bar{v}dx \pm i\int_{\partial\Omega}w\bar{v}dS=0.
    \end{equation*}
    Let $v=w$. The real and imaginary part should be $0$, we deduce that $w=0$ on the boundary, and so is its outer normal derivative. Then by the unique continuation property with Cauchy data for elliptic operators \cite{alessandrini2009stability}, we conclude $w=0$ in $\Omega$. This shows the uniqueness of  both problems, and thus the existence, and the inverse bound follows by \textbf{bounded inverse theorem}, see, for example in \cite[Corollary 2.7]{brezis}. In the specific case when $P\ge 0$, we choose $v = u$ in the weak formulation and take real and imaginary part to get the estimates
    \begin{equation*}
        ||\nabla u||_\Omega \le C||f||_{(H^1(\Omega))'},
    \end{equation*}
    and
    \begin{equation*}
        ||u||_{\partial\Omega} \le C||f||_{(H^1(\Omega))'}
    \end{equation*}
    Combining the two equations and by \textbf{Poincare's inequality} we have
    \begin{equation*}
        ||u||_{H^1(\Omega)}\le C||f||_{(H^1(\Omega))'}.
    \end{equation*}
\end{proof}

\begin{proof}[Proof of \thref{co2}]
    We only need to prove the inequality for $u\in H^2(\Omega)$ and use a density argument for $u\in H^1(\Omega)$. Let $B_0\subset B_1\subset B_2\subset B_3\subset B_4\subset \Omega$, $\psi\in C_0^\infty(B_4)$ be a cut-off function in $B_4$ and equals $1$ in $B_3$, consider the following equation,
    \begin{equation}\label{dir}
       \left\{\begin{array}{rlc}
        -\Delta w+Pw & =\psi(-\Delta u+Pu) & \texttt{in $B_4$,}   \\
        \partial w+iw & =0 & \texttt{on $\partial B_4$.}
        \end{array}\right.
    \end{equation}
    According to \thref{aux}, (\ref{dir}) has a unique solution and thus satisfies the inverse bound:
    \begin{equation}\label{invbnd}
    \begin{aligned}
      ||w||_{H^1(B_3)}&\le C(P)||\psi(-\Delta u+Pu)||_{(H^1(B_4))'}\\
      &=C(P)\sup_{v\in H^1(B_4)} <-\Delta u+Pu,\psi \frac{v}{||v||_{H^1(B_4)}}>\\ 
      &\le C(P)||-\Delta u+Pu||_{H^{-1}(B_4)} ,
    \end{aligned}
    \end{equation}
    where $<\cdot,\cdot>$ represents the $H^{-1}-H^1$ duality. Note that $C(P)$ is independent of $P$ if $P\ge 0$, in which case the auxiliary problem (\ref{dir}) has an exact inverse bound independent of $P$.
    Consider the difference $v=u-w$.  $-\Delta v+Pv$ vanishes in $B_3$. According to \thref{tbi}, we have 
    \begin{equation}
        ||v||_{H^1(B_2)}\le C(P)||v||_{H^1(B_0)}^\kappa ||v||_{H^1(B_3)}^{1-\kappa}.
    \end{equation}
 $u$ can then be bounded by 
    \begin{equation*}
        \begin{aligned}
            ||u||_{H^1(B_2)} &\le ||v||_{H^1(B_2)}+||w||_{H^1(B_2)}\\
            & \le C(P)(||u||_{H^1(B_0)}+||w||_{H^1(B_0)})^\kappa (||u||_{H^1(B_3)}+||w||_{H^1(B_3)})^{1-\kappa}\\
            &+C(P)||-\Delta u+Pu||_{H^{-1}(\Omega)}\\
            & \le C(P)(||u||_{H^1(B_0)}+||-\Delta u+Pu||_{H^{-1}(\Omega)})^\kappa (||u||_{H^1(B_3)}+||-\Delta u+Pu||_{H^{-1}(\Omega)})^{1-\kappa}.
        \end{aligned}
    \end{equation*}
    To bound $||u||_{H^1(B_0)}$ using the $L^2(B_1)$ norm for the data, let $\chi\in C_0^\infty(B_1)$ satisfy $\chi=1$ in $B_0$, and consider
    \begin{equation}
        \left\{\begin{array}{rlc}
        -\Delta \chi u+P\chi u & =\chi(-\Delta u+Pu)+[-\Delta,\chi]u & \texttt{in $B_1$}   \\
         \partial_n\chi u+i\chi u & =0 & \texttt{on $\partial B_1$}
        \end{array}\right.,
    \end{equation}
    where $[-\Delta,\chi] := -\Delta\chi+\chi\Delta$. We have by the same arguments as for (\ref{invbnd})
    \begin{equation}\label{eq:H1commutebounf}
    ||u||_{H^1(B_0)}\le ||\chi u||_{H^1(B_1)} \le C(P)(||\chi(-\Delta u+Pu)||_{(H^1(B_1))'}+||[-\Delta,\chi]u||_{(H^1(B_1))'}).
    \end{equation}
    Noticing that $\chi$ makes boundary terms vanish in integration by part, i.e. formally,
    \begin{equation}
        [-\Delta, \chi]u = -u\Delta\chi-2\nabla u\cdot\nabla\chi
    \end{equation}
    we have for all $w\in H^1(B_1)$, 
    \begin{equation}\label{comm}
    \begin{aligned}
        &\int_{B_1}uw\Delta \chi dx+2\int_{B_1}w\nabla u\cdot\nabla \chi dx \\
        =&\int_{B_1}uw\Delta \chi dx+2\int_{\partial B_1}wu\nabla \chi\cdot ndS-2\int_{B_1}(uw\Delta \chi+u\nabla\chi\cdot\nabla w)dx\\
        =&-\int_{B_1}uw\Delta \chi dx-2\int_{B_1}u\nabla\chi\nabla wdx\\
        \le &C(\chi)||u||_{B_1}||w||_{H^1(B_1)}.
    \end{aligned}
    \end{equation}
    Applying the estimate of (\ref{comm}) in (\ref{eq:H1commutebounf}), we have
    \begin{equation*}
        ||u||_{H^1(B_0)}\le C(P)(||-\Delta u+Pu||_{H^{-1}(B_1)}+||u||_{L^2(B_1)}).
    \end{equation*}
    Doing the same for $u$ with $B_2$ and $B_3$ we obtain the stated result.\\
\end{proof}
Now we give the proof of \thref{ultraweak}. The proof is adapted from \cite[Theorem 2.2]{monsuur2024ultra} for Poisson's equation. We give a modified version here for the reader's convenience.
\begin{proof}[Proof of \thref{ultraweak}]

Let $\mathcal{H}(\Omega)\subset L^2(\Omega)$ be the range of $-\Delta+P$ defined on $H^2_0(\Omega)$. We claim that $\mathcal{H}(\Omega)$ is a closed subspace of $L^2(\Omega)$. This can be justified by showing the following estimate: $\forall u\in H^2_0(\Omega)$,
\begin{equation*}
    ||u||_{H^2(\Omega)} \le C(P) ||(-\Delta+P)u||_\Omega.
\end{equation*}
Note that in \thref{ultraweak}, we do not assume our domain is convex Polygonal or smooth. Thus we need to extend an outer ball such that the elliptic regularity holds. Let $Z$ be a ball such that $\Omega\subset Z$, we could extend $u$ and $P$ in $Z$ by $0$. We denote the extended functions still $u$ and $P$. Interior elliptic regularity implies:
\begin{equation*}
\begin{aligned}
    ||u||_{H^2(\Omega)}& \lesssim ||(-\Delta+P)u||_Z + ||u||_Z \\
    &= ||(-\Delta+P)u||_\Omega + ||u||_\Omega.
    \end{aligned}
\end{equation*}
The Schrödinger operator on $H^2_0(\Omega)$ is injective by the unique continuation property. Moreover, the inclusion $H^2(\Omega)\subset L^2(\Omega)$ is compact. The Petree–Tartar lemma, see e.g. \cite[Lemma A.20]{ern2021finite}, implies that
\begin{equation*}
    ||u||_{H^2(\Omega)} \lesssim ||(-\Delta+P)u||_\Omega,
\end{equation*}
which justifies our claim.

Let $w\in L^2(\Omega)$, and let $\hat{w}\in \mathcal{H}(\Omega)$ denote its orthogonal projection on $\mathcal{H}(\Omega)$. Then we have $\forall v\in H^2_0(\Omega)$,
\begin{equation*}
    \int_\Omega w (-\Delta+P)v = \int_\Omega \hat{w}(-\Delta+P)v.
\end{equation*}

Letting $-\Delta\phi +P\phi = \hat{w}$, we have the bound
\begin{equation*}
    ||\hat{w}||^2_\Omega \le ||(-\Delta+P)w||_{H^{-2}(\Omega)}||\phi||_{H^2(\Omega)}\le C(P)||(-\Delta+P)w||_{H^{-2}(\Omega)}||\hat{w}||_\Omega,
\end{equation*}
so that $||\hat{w}||_\Omega \le C(P)||(-\Delta+P)w||_{H^{-2}(\Omega)}$. Now take $\tilde w=w-\hat{w}$, then $\forall v\in H^2_0(\Omega)$
\begin{equation*}
    \int_\Omega\tilde w \Delta vdx = \int_\Omega P\tilde{w}vdx.
\end{equation*}
By elliptic regularity for the Laplace operator in the distribution space (See, e.g. \cite[Theorem 9.26]{folland1999real}), since $P\tilde w\in L^2(\Omega)$, we have that $\tilde w\in H^2_{loc}(\Omega)$. Moreover, $-\Delta \tilde w+P\tilde w=0$.

Let $B_1\subset B_2\subset B_3\subset \Omega$, we have by \thref{co2}
\begin{equation*}
    ||\tilde w||_{B_2} \le ||\tilde w||_{B_1}^\kappa||\tilde w||_{B_3}^{1-\kappa}.
\end{equation*}
To conclude we consider the $L^2(B_2)$-norm of $w$ and use that $w = \tilde w + \hat w$ and the above bounds in the following manner
\begin{equation*}
\begin{aligned}
    ||w||_{B_2} &\le ||\hat{w}||_{\Omega} + ||\tilde w||_{B_2}\\
    &\le ||(-\Delta+P)w||_{H^{-2}(\Omega)} + ||\tilde w||_{B_1}^\kappa||\tilde w||_{B_3}^{1-\kappa}\\
    &\le ||(-\Delta+P)w||_{H^{-2}(\Omega)} + (||w||_{B_1}+||\hat{w}||_{B_1})^\kappa (||w||_{B_3}+||\hat{w}||_{B_3})^{1-\kappa}\\
    &\le (||w||_{B_1}+||(-\Delta+P)w||_{H^{-2}(\Omega)})^\kappa ((||w||_{\Omega}+||(-\Delta+P)w||_{H^{-2}(\Omega)})^{1-\kappa}.
    \end{aligned}
\end{equation*}
\end{proof}
\section{Proof of \thref{L2error}}
\label{sec:proofL2error}
Before proving the following estimate, we first state two inequalities for semiclassical Soblev norm, see, for example, \cites{dos2009limiting,zworski2022semiclassical}. Let
    $$J^s = (1-\hbar^2\Delta)^\frac{s}{2}$$ 
    be the Bessel potential, we will use two estimates with semiclassical Sobolev norm. Let $\theta$, $\eta\in C_0^\infty(\mathbb{R}^n)$, where $\eta=1$ near supp($\theta$), and $A,B$ be two semiclassical pseudodifferential operators of orders $s$, $m$, respectively. Then for all $p,q,N\in \mathbb{R}$, there exists $C$, such that 
    \begin{equation}\label{eq1}
        ||[A,B]v||_{H^p_{scl}}\le C\hbar||v||_{H^{p+s+m-1}_{scl}},
    \end{equation}
    \begin{equation}\label{eq2}
        ||(1-\eta)A\theta v||_{H^p_{scl}}\le C\hbar^N||v||_{H^q_{scl}}.
    \end{equation}

To prove \thref{L2error}, we need the following lemma:
\begin{lemma}\thlabel{sft}
    Let $v\in C_0^\infty(B(x_0,R)\backslash B(x_0,r))$ for some $R,r$, $\hbar=\frac{1}{\tau}$,  $F:=-\hbar^2e^\frac{\phi}{\hbar}\Delta e^\frac{\phi}{\hbar}$ be a semiclassical pseudodifferential operator, where $\tau$ and $\phi$ is defined as in \thref{car}, and $P$ be the bounded potential. Then, there exists a constant $C$, such that
    \begin{equation*}
        \sqrt{\hbar}||v||_{L^2}\le C||(F+\hbar^2P)v||_{H^{-1}_{scl}},
    \end{equation*}
    where $||\cdot||_{H^{-1}_{scl}}$ denotes semiclassical Sobolev norm, defined by 
    \begin{equation*}
        ||u||_{H^{-1}_{scl}} :=||(1-\hbar^2\Delta)^\frac{s}{2}u||_{L^2(\mathbb{R}^n)}
    \end{equation*}
\end{lemma}
\begin{proof}[Proof of \thref{sft}]
    A similar proof can be referred to \cite{burman2020stabilized} for the convection-diffusion equation, but we will give it for the readers' convenience. 
    Since $v\in C_0^\infty(B(x_0,R)\backslash B(x_0,r))$, we construct the same way as in \thref{tbi}, and by Carleman estimate and some simple algebra, we have
    \begin{equation*}
        C\int_{\mathbb{R}^n}|e^\frac{\phi}{\hbar}\Delta e^{-\frac{\phi}{\hbar}}v|^2dx\ge        \int_{\mathbb{R}^n}(\hbar^{-1}-1)|\nabla v|^2+(\hbar^{-3}-\hbar^{-2})v^2dx.
    \end{equation*}
    Adding potential term on both sides and rescaling by $\hbar^4$ with $\hbar <<1$ we obtain 
    \begin{equation*}
        C\int_{\mathbb{R}^n}|(F+\hbar^2P)v|^2dx\ge \int_{\mathbb{R}^n}\hbar^3|\nabla v|^2+(\hbar-C\hbar^4P^2)v^2dx.
    \end{equation*}
    Note that $||v||_{H^1_{scl}}=\int_{\mathbb{R}^n}\hbar^2|\nabla v|^2+v^2dx$. By letting 
    \begin{equation}\label{con1}
         \hbar^3<<\frac{1}{C||P||^2_\infty}
    \end{equation}
   
    the right-hand term with potential can be absorbed and thus we have
    \begin{equation}\label{l2es}
        \sqrt{\hbar}||v||_{H^1_{scl}}\le C||(F+\hbar^2P)v||_{L^2}
    \end{equation}
    Now let $\theta,\eta\in C_0^\infty(\mathbb{R}^n)$ satisfying that $\theta=1$ around support of $v$ and $\eta=1$ near support of $\theta$. Then by (\ref{eq2}), for $\hbar<\hbar_0$ for $\hbar_0$ depending only on geometry, we have,
    \begin{equation*}
        \begin{aligned}
            ||v||_{L^2}&=||J^{-1}v||_{H^1_{scl}}\\
            &\le ||\eta J^{-1}v||_{H^1_{scl}}+||(1-\eta) J^{-1}\theta v||_{H^1_{scl}}\\
            &\le C||\eta J^{-1}v||_{H^1_{scl}}
        \end{aligned}
    \end{equation*}
    Note that $\eta J^{-1}v$ is also compactly supported. By (\ref{l2es}) we have
    \begin{equation*}
        \begin{aligned}
            \sqrt{\hbar}||v||_{L^2}\le C\sqrt{\hbar}||\eta J^{-1}v||_{H^1_{scl}}\le C||(F+\hbar^2P)\eta J^{-1}v||_{L^2}
        \end{aligned}
    \end{equation*}
    Now commute $(F+\hbar^2P)$ and $\eta J^{-1}$, we have the estimate
    \begin{equation*}
    \begin{aligned}
        ||[(F+\hbar^2P), \eta J^{-1}]v||_{L^2}&=||[F,\eta J^{-1}]v||_{L^2}+||[\hbar^2P,\eta J^{-1}]v||_{L^2}\\
        &\le C\hbar||v||_{L^2}+C\hbar^2||P||_{L^\infty}||v||_{H^{-1}_{scl}},
        \end{aligned}
    \end{equation*}
    where the first estimate is from (\ref{eq2}) and the second estimate is simply from the definition of $||\cdot||_{H^{-1}_{scl}}$. Note that the constant $C$ above onle depends on geometry, we can make $\hbar_0$ small enough to absorb the first term, and by (\ref{con1}) the second term can also be absorbed. Thus
    \begin{equation}\label{sftes}
        \sqrt{\hbar}||v||_{L^2}\le C||J^{-1}(F+\hbar^2 P)v||_{L^2}=C||(F+\hbar^2 P)v||_{H^{-1}_{scl}}.
    \end{equation}

\end{proof}
\begin{proof}[Proof of \thref{L2error}]
    Let $r_0<r_{1,}<r_1$, $\chi\in C^\infty_0(B_4\backslash B_0)$ and $\chi=1$ in $B_3\backslash B_{1'}$, $\psi\in C^\infty_0(B_1\cup\Omega\backslash B_2)$ and $\psi=1$ in $(B_4\backslash B_3)\cup (B_{1'}\backslash B_0)$, then we have $\psi=1$ while $[F,\chi]\neq 0$. Let $u\in C^\infty(\Omega)$, and we obtain by (\ref{eq2})
    \begin{equation*}
        ||[F+\hbar^2P,\chi]e^\frac{\phi}{\hbar}u||_{H^{-1}_{scl}}\le C||[F,\chi]\psi e^\frac{\phi}{\hbar}u||_{H^{-1}_{scl}}\le C\hbar ||\psi e^\frac{\phi}{\hbar}u||_{L^2}.
    \end{equation*}
    Notice that $\chi e^\frac{\phi}{\hbar}u\in C_0^\infty (B_4)$, which satisfies the condition in \thref{sft}. Thus combining (\ref{eq2}) and (\ref{sftes}) we obtain
    \begin{equation*}\begin{aligned}
        \sqrt{\hbar}||\chi e^\frac{\phi}{\hbar}u||_{L^2}&\le C||(F+\hbar^2P)\chi e^\frac{\phi}{\hbar}u||_{H^{-1}_{scl}}\\
        &\le C\hbar^2||\chi e^\frac{\phi}{\hbar}(-\Delta+P)u||_{H^{-1}_{scl}}+C\hbar||\psi e^\frac{\phi}{\hbar}u||_{L^2}.
        \end{aligned}
    \end{equation*}
    Now according to the norm bound $||\cdot||_{H^{-1}_{scl}}\le C\hbar^{-2}||\cdot||_{H^{-1}}$ and the support of $\chi$ and $\psi$, a local version of estimate is obtained:
    \begin{equation}\label{eq3}
         \sqrt{\hbar}||e^\frac{\phi}{\hbar}u||_{L^2(B_3)}\le  \sqrt{\hbar}||e^\frac{\phi}{\hbar}u||_{L^2(B_1)}+C|| e^\frac{\phi}{\hbar}(-\Delta+P)u||_{H^{-1}(\Omega)}+C\hbar||\psi e^\frac{\phi}{\hbar}u||_{L^2(\Omega)}.
    \end{equation}
    Notice the last term on the right-hand side can be separated into three components, i.e. $B_1,B_3\backslash B_2, \Omega\backslash B_3$. Notice that the $B_1$ part can be absorbed by $\sqrt{\hbar}||e^\frac{\phi}{\hbar}u||_{L^2(B_1)}$ and the $B_3\backslash B_2$ part can be absorbed by the left-hand side for small $\hbar$. Thus we have by (\ref{eq3}),
    \begin{equation*}
    \begin{aligned}
        ||u||_{L^2(B_2)}\le &Ce^{-\frac{\Phi(r_2)}{\hbar}}\left(||u||_{L^2(B_1)}+\hbar^{-\frac{1}{2}}|| (-\Delta+P)u||_{H^{-1}(\Omega)}\right)\\
        &+C\hbar^{-\frac{3}{2}} e^{-\frac{\Phi(r_2)-\Phi(r_3)}{\hbar}}||u||_{L^2(\Omega)}.
        \end{aligned}
    \end{equation*}
    Then by absorbing the polynomial term by exponential term and following the same procedure as in \thref{tbi}, the conclusion follows.
    
\end{proof}





\bibliographystyle{amsplain}
\bibliography{bibliography.bib}

\end{document}